\documentclass[11pt, reqno]{amsart}
\usepackage{amsmath,amsthm,amssymb,mathrsfs,graphicx, enumitem, hyperref, textcomp}
\usepackage{soul}

\usepackage[all,cmtip,pdf]{xy}

\usepackage{xcolor}
\usepackage{tikz}
\usetikzlibrary{positioning}
\usepackage{etoolbox}
\let\bbordermatrix\bordermatrix
\patchcmd{\bbordermatrix}{8.75}{4.75}{}{}

\newcommand{\U}{\mathcal{U}}

\newcommand{\E}{\mathcal{E}}
\newcommand{\F}{\mathcal{F}}
\newcommand{\Ell}{\textup{Ell}}

\newcommand{\id}{\mathrm{id}}
\newcommand{\Aut}{\mathrm{Aut}}

\newcommand{\diag}{\mathrm{diag}}

\newcommand{\dlim}{\underset{\longrightarrow}{\lim}}

\theoremstyle{plain}
\newtheorem{theorem}{Theorem}[section]
\newtheorem{lemma}[theorem]{Lemma}
\newtheorem{corollary}[theorem]{Corollary}
\newtheorem{proposition}[theorem]{Proposition}

\newtheorem*{theorem*}{Theorem}
 \newtheorem*{lemma*}{Lemma}
\theoremstyle{definition}
\newtheorem{definition}[theorem]{Definition}
\newtheorem{example}[theorem]{Example}
\newtheorem{examples}[theorem]{Examples}

\newtheorem{remark}[theorem]{Remark}

\theoremstyle{remark}

\theoremstyle{plain}
\newtheorem*{thm}{Theorem}

\author[Forough, Jeong, Strung]{Marzieh Forough \and Ja A Jeong \and Karen R. Strung}

\address{Department of Applied Mathematics, Faculty of Information Technology\\ Czech Technical University in Prague\\ Th\'akurova 9 \\160 00, Prague 6, Czech Republic}
\email{foroumar@fit.cvut.cz}

\address{Department of Mathematical Sciences and Research Institute of Mathematics\\
Seoul National University\\
Seoul 08826, Korea}
\email{jajeong@snu.ac.kr}

\address{Department of Abstract Analysis\\ Institute of Mathematics, Czech Academy of Sciences, \v{Z}itn\'a 25, 115 67 Prague 1, Czech Republic}
\email{strung@math.cas.cz}

\title[Recursive subhomogeneity of orbit breaking subalgebras]{Recursive subhomogeneity of orbit-breaking subalgebras of $\mathrm{C}^*$-algebras associated to minimal homeomorphisms twisted by line bundles}

\thanks{KRS is funded by \mbox{RVO: 67985840}. Part of this work was carried out when KRS was funded by GA\v{C}R project 20-17488Y and when JAJ was partially supported by NRF 2018R1D1A1B07041172.}

\begin{document}

\begin{abstract} In this paper, we construct recursive subhomogeneous decompositions for the Cuntz--Pimsner algebras obtained from breaking the orbit of a minimal Hilbert $C(X)$-bimodule at a closed subset $Y \subset X$ with non-empty interior. This generalizes the known recursive subhomogeneous decomposition for orbit-breaking subalgebras of crossed products by minimal homeomorphisms. 
\end{abstract}

\maketitle

\tableofcontents


\section{Introduction}

Recursive subhomogeneous (RSH) algebras are a particularly tractable class of unital subhomogeneous $\mathrm C^*$-algebras, which are given by iterated pullbacks of the form $C(X,M_n)$, for $X$ a compact Hausdorff space. This class of $\mathrm{C}^*$-algebras is rich and includes homogeneous $\mathrm{C}^*$-algebras and their direct sums, section algebras of locally trivial continuous fields over compact spaces with matrix fibres, the dimension drop intervals and matrix algebras over them, and the subalgebras of crossed products by minimal homeomorphisms given by breaking the orbit of the homeomorphism at a closed subset with non-empty interior (see \cite{LinQPhil:KthoeryMinHoms} and Section~\ref{sec:rsh}).  

Introduced by Phillips in~\cite{Phillips:recsub}, the similarity of recursive subhomogeneous algebras to direct sums of homogeneous $\mathrm{C}^*$-algebras allowed for the extension of constructions required for the study of $K$-theory, stable, real, and decomposition rank, and simple inductive limits of interest to the Elliott classification program \cite{Phillips:recsub, Phi:CancelSRDirLims, Winter:subhomdr}.  Indeed, when a simple unital $\mathrm{C}^*$-algebra $A$ can be modelled as an inductive limit of separable recursive subhomogeneous $\mathrm{C}^*$-algebras with slow dimension growth, then it is classifiable \cite{EllGongLinNiu2017}. In fact, inductive limits of RSH algebras are known to exhaust the Elliott invariant \cite{Ell:invariant}. 

If a $\mathrm{C}^*$-algebra $A$ is $^*$-isomorphic to an RSH algebra $B$, then we say that $B$ is an RSH decomposition of $A$. In general, an RSH decomposition (if it exists) may not be unique.   In this paper, we provide RSH decompositions for certain \emph{orbit-breaking} subalgebras of Cuntz--Pimsner algebras of Hilbert bimodules over commutative $\mathrm{C}^*$-algebras. 

If $A \cong C(X)$ for a compact metrizable space $X$, then any Hilbert bimodule over $A$ which is finitely-generated projective as a right Hilbert $A$-module and is left full, is necessarily given by a right Hilbert module of sections of a line bundle $\mathscr V$ over $X$ with left action given by the right action of $C(X)$ after composing with a homeomorphism $\alpha : X \to X$ \cite[Proposition 3.7]{AAFGJSV2024}. We say that the homeomorphism $\alpha$ is twisted by the line bundle $\mathscr{V}$ and denote such a module by $\Gamma(\mathscr{V}, \alpha)$ and the corresponding Cuntz--Pimsner algebra as $\mathcal{O}(\Gamma(\mathscr{V}, \alpha))$. (Since $\Gamma(\mathscr{V}, \alpha)$ is a Hilbert bimodule, one could equivalently construct the crossed product $C(X) \rtimes_{\Gamma(\mathscr{V}, \alpha)} \mathbb{Z}$ in the sense of \cite{AEE:Cross}. We use the Cuntz--Pimsner notation because it is slightly more concise.) When the vector bundle $\mathscr{V}$ is trivial, then $\mathcal{O}(\Gamma(\mathscr{V}, \alpha))$ is precisely the crossed product  $C(X)\rtimes_\alpha \mathbb Z$ of $C(X)$ by the homeomorphism $\alpha$. Now let $Y \subset X$ be a closed non-empty subset. The orbit-breaking Hilbert $C(X)$-bimodule, $C_0(X \setminus Y) \Gamma(\mathscr{V}, \alpha)$, is the Hilbert subbimodule of $\Gamma(\mathscr{V}, \alpha)$ where we have restricted the left multiplication to those functions in $C(X)$ which vanish on the set $Y$. We call the Cuntz--Pimsner algebra $\mathcal{O}(C_0(X \setminus Y)\Gamma(\mathscr{V}, \alpha))$ an \emph{orbit-breaking algebra}. When the vector bundle is trivial, $\mathcal{O}(C_0(X \setminus Y)\Gamma(\mathscr{V}, \alpha))$  is isomorphic to an orbit-breaking subalgebra of the usual crossed product $C(X) \rtimes_\alpha \mathbb{Z}$.

From their introduction by Putnam in the case of Cantor minimal systems \cite{Putnam:MinHomCantor}, orbit-breaking algebras have played a critical role in understanding properties and structures of crossed products such as $\mathcal{Z}$-stability and stable rank \cite{LinPhi:MinHom, TomsWinter:PNAS, TomsWinter:minhom, Sun:CantorTorus, Str:XxSn, Win:ClassCrossProd, Lin:MinDynOdd, alboiu-lutley}. They also are interesting $\mathrm{C}^*$-algebras in their own right, as sufficient conditions are known for when they are simple and $\mathcal{Z}$-stable. Here, $\mathcal{Z}$-stability refers to the property of tensorial absorption of the Jiang--Su algebra $\mathcal{Z}$---a simple, separable, unital, nuclear, projectionless, infinite-dimensional $\mathrm{C}^*$-algebra with unique tracial state and with the same $K$-theory as $\mathbb{C}$ \cite{JiaSu:Z}. 
That is, a $\mathrm{C}^*$-algebra $A$ is \emph{$\mathcal{Z}$-stable} if $A \otimes \mathcal{Z} \cong A$. As such, they provide dynamical models of  $\mathrm{C}^*$-algebras that are covered by the classification theorem, which says that unital, simple, separable, nuclear $\mathrm{C}^*$-algebras which satisfy the Universal Coefficient Theorem (UCT) and are $\mathcal{Z}$-stable, are classified, up to $^*$-isomorphism, by $K$-theory and tracial data (see for example \cite{CETWW, BBSTWW:2Col, EllGonLinNiu:ClaFinDecRan, GongLinNiue:ZClass, GongLinNiue:ZClass2, TWW, CarGabeSchafTikWhi2023}). The UCT is, loosely speaking, a technical requirement for classification results that allows one to lift information from $K$-theory to $KK$-theory. All known nuclear $\mathrm{C}^*$-algebras satisfy the UCT, but it remains an open problem to determine whether nuclearity implies the UCT (see \cite{BarLi:UCTI, WinterQDQUCT} for a discussion).

We say that a  unital, simple, separable, nuclear, $\mathcal{Z}$-stable $\mathrm{C}^*$-algebra which satisfies the UCT is \emph{classifiable}. It was shown in \cite{DPS:OrbitBreaking} that orbit-breaking subalgebras of crossed products can realize a wide range of $K$-theory---much wider than the crossed products themselves---and hence a wide range of stably finite, classifiable $\mathrm{C}^*$-algebras. The introduction of a non-trivial line bundle has the potential to expand this range, possibly exhausting the range of the Elliott invariant in the stably finite case.

In~\cite{AAFGJSV2024}, the authors introduced orbit-breaking subalgebras of the Cuntz--Pimsner algebra of a minimal homeomorphism $\alpha$ over a compact metrizable space $X$, twisted by a vector bundle $\mathscr{V}$. When $\mathscr{V}$ is a line bundle and the subspace $Y \subset X$ at which the orbit broken is sufficiently small (for example, a finite set of points), it was proved that $\mathcal{O}(C_0(X\setminus Y) \Gamma(\mathscr{V}, \alpha))$ is a centrally large subalgebra of $\mathcal{O}(\Gamma(\mathscr{V}, \alpha))$ as defined by Archey and Phillips in \cite{ArchPhil:SR1}. This implies that $\mathcal{O}(C_0(X\setminus Y) \Gamma(\mathscr{V}, \alpha))$ is simple and stably finite. Moreover, if the underlying space has finite covering dimension, it was shown that orbit-breaking subalgebras are $\mathcal{Z}$-stable. To show this, one uses results about the Rokhlin dimension of a homeomorphism \cite{Sza:RokDimZd} on a finite-dimensional metric space to first prove that the Cuntz--Pimsner algebra itself is $\mathcal{Z}$-stable. Appealing to the fact that $\mathcal{O}(C_0(X\setminus Y) \Gamma(\mathscr{V}, \alpha))$ is centrally large, we conclude that is also $\mathcal{Z}$-stable.

This implies, via the classification machinery, that $\mathcal{O}(C_0(X\setminus Y) \Gamma(\mathscr{V}, \alpha))$ is an inductive limit of RSH algebras. However, the result doesn't say anything about the RSH building blocks in the inductive system. Even when $\mathscr{V}$ is trivial, and we are in the setting of crossed products and their orbit-breaking algebras, this result does not recover the approximately RSH structure given by Q. Lin in an unpublished preprint \cite{QLin:Ay} (see \cite{LinQPhil:KthoeryMinHoms} for an outline of the construction). It also gives no information about what happens when the subset $Y$ has non-empty interior (in this case the orbit-breaking algebra will not be simple). 

The purpose of this paper is to exhibit an RSH decomposition for the Cuntz--Pimsner algebra $\mathcal{O}(C_0(X \setminus Y) \Gamma(\mathscr{V}, \alpha))$ when $Y \subset X$ is a closed subset with non-empty interior.

Such a construction will allow us to investigate further properties of the $\mathrm C^*$-algebras associated to minimal homeomorphisms twisted by line bundles and their orbit-breaking subalgebras. In particular, a follow-up paper by the authors~\cite{FJS} will show that if $(X, \alpha)$ is a minimal dynamical system with \emph{mean dimension zero} (a dynamical generalization of covering dimension introduced by Gromov \cite{GromovMD} and developed by Lindenstrauss and Weiss \cite{LindWeiss:MTD}),  $\mathscr{V}$ is a line bundle, and $Y \subset X$ is sufficiently small, then both $\mathcal{O}(C_0(X \setminus Y)\Gamma(\mathscr{V}, \alpha))$ and $\mathcal{O}(\Gamma(\mathscr{V}, \alpha))$ are $\mathcal{Z}$-stable and hence classifiable. In subsequent work, further properties such as $K$-theory and stable rank will be studied.

The structure of the paper is as follows. Section~\ref{sec:prelims} provides a short introduction to Hilbert bimodules and Cuntz--Pimsner algebras, and related preliminaries.  

In Section~\ref{sec:TP} the tensor products of Hilbert $C(X)$-bimodules are described in terms of tensor products of line bundles and the composition of homeomorphisms. We make a considerable use of this description in this work.

In Section~\ref{sec:OBA}, we recall the definition and  properties of orbit-breaking algebras, which are the Cuntz--Pimsner algebras of orbit-breaking bimodules. 

In Section~\ref{sec:rsh}, we briefly recall the recursive subhomogeneous structure of orbit-breaking algebras in the case of crossed products by minimal homeomorphisms, and we outline the construction of the RSH decomposition of orbit-breaking algebras in our setting. 

Endomorphism bundles of vector bundles are the focus of Section~\ref{sec:endomorphism bdl}. Vector bundles constructed from the line bundles in Section~\ref{sec:TP} are used in this framework to construct the building blocks of the recursive subhomogeneous algebras.

Section~\ref{sec:embedding} establishes an embedding of the orbit-breaking algebras into a direct sum of section algebras of endomorphism bundles. 

The final section is devoted to the RSH decomposition for the range of the embedding obtained in {Section~\ref{sec:embedding}. Indeed, we show that  $\mathcal{O}(\E_Y)$  is isomorphic to the iterated pullback
\begin{equation*}
    \Gamma(\mathscr{M}_1) \oplus_{\Gamma(\mathscr{M}_2|_{\overline{Y_2}\setminus Y_2})} \oplus \dots \oplus_{\Gamma(\mathscr{M}_K|_{\overline{Y_K}\setminus Y_K})} \Gamma(\mathscr{M}_K),
\end{equation*} 
where $\Gamma(\mathscr{M}_k)$ are section algebras of locally trivial matrix bundles whose bundle structure is induced by the line bundle $\mathscr{V}$.  This implies that the orbit-breaking algebras are recursive subhomogeneous algebras and allows us to conclude our main theorem, \textbf{Theorem~\ref{maintheorem}}:

\begin{thm}
 Let $X$ be an infinite compact metric space, $\alpha : X \to X$ a minimal homeomorphism, $\mathscr{V}$ a line bundle over $X$ and $Y \subset X$ a closed subset with non-empty interior. Let $r_1 < \dots < r_K$ denote the distinct first return times to $Y$. The orbit-breaking algebra $\mathcal O (C_0(X \setminus Y)\Gamma(\mathscr{V}, \alpha))$ has a recursive subhomogeneous decomposition $B$ such that 
 \begin{enumerate}
     \item the length $l$ of $B$ is at least $K$,
     \item the base spaces $Z_1, \dots, Z_l$ have dimension bounded by $X$, and there is $l_0= 0 < l_1 < l_2 < \dotsm < l_K = l$ such that 
     \[
    \bigcup_{j=l_{k-1}+1}^{l_k} Z_j = \overline{\{ y \in Y \mid r(y) = r_k\}},
     \]
     for every $1 \leq k \leq K$,
     \item the matrix sizes are exactly $r_1, \dots, r_K$.
 \end{enumerate}
\end{thm}

From this, we can recover the classification results from \cite{AAFGJSV2024}, see \textbf{Corollary~\ref{cor:mainthmCor}}. We end with some remarks about how our RSH construction generalizes that of \cite{QLin:Ay}, and we discuss how to obtain more general results concerning classification and stable rank when the base space $X$ is not necessarily of finite dimension.

\subsubsection*{Acknowledgments} The authors would like to thank Gi Hyun Park, Maria Grazia Viola, Magdalena Georgescu and Dawn Archey for helpful discussions. This project was initiated as part of the workshop 18w5168  \emph{Women in Operator Algebras} at the Banff International Research Station in November 2018.


\section{Preliminaries} \label{sec:prelims}
In this section we review the basics of Hilbert bimodules, with particular focus on Hilbert $C(X)$-bimodules. We then recall the construction of Cuntz--Pimsner algebras, as well as some useful facts that we will require in subsequent sections.

\subsection{Hilbert bimodules}
We are particularly interested in Hilbert $C(X)$-bimodules, but we begin with recalling the definition and relevant results about Hilbert bimodules over arbitrary $\mathrm C^*$-algebras.

\begin{definition} Let $A$ and $B$ be $\mathrm{C}^*$-algebras. A \emph{Hilbert $A$-$B$ bimodule} $\E$
is a left $A$-module, right $B$-module equipped with
\begin{itemize}[left=0pt, label=\tiny$\bullet$]
    \item a right $B$-valued inner product $\langle \cdot, \cdot \rangle_\E$  giving $\E$ a right Hilbert $B$-module structure, and
    \item a left $A$-valued inner product $_\E \langle \cdot , \cdot \rangle$  giving $\E$ a left Hilbert $A$-module structure,
\end{itemize}
such that compatibility condition
\[ \xi_1  \langle \xi_2, \xi_3 \rangle_\E = \, _\E \!\langle \xi_1, \xi_2 \rangle  \xi_3, \]
is satisfied for every $\xi_1, \xi_2, \xi_3 \in \E$.
If $A = B$ we call $\E$ a Hilbert $A$-bimodule. 
\end{definition}

A Hilbert $A$-$B$-bimodule $\E$ is \emph{left}, respectively \emph{right}, \emph{full} if $\overline{_\E\langle \E, \E \rangle} \  = A$, respectively $\overline{\langle \E, \E \rangle_\E}\ = B$. Here $\overline{\langle \E, \E\rangle}$ denotes the closed linear span of $\{ \langle \xi_1, \xi_2 \rangle \mid \xi_1, \xi_2\in \E\}$.

Any Hilbert $A$-$B$-bimodule may also be viewed as a special case of an $A$-$B$ \emph{$\mathrm C^*$-correspondence}, that is, a right Hilbert $B$-module equipped with a \emph{structure map} $\varphi_\E : A \to \mathcal{L}(\E)$, which is a $^*$-homomorphism from $A$ into the $\mathrm C^*$-algebra of adjointable operators $\mathcal{L}(\E)$. The structure map of a Hilbert $A$-$B$-bimodule is just given by left multiplication. Note that not every $A$-$B$ $\mathrm C^*$-correspondence can be given the structure of a Hilbert $A$-$B$-bimodule.

\subsubsection{\texorpdfstring{Tensor products of  Hilbert bimodules}{Tensor products of  Hilbert C(X)-bimodules}}
Let $A, B$ and  $C$ be $\mathrm C^*$-algebras. Given a Hilbert $A$-$B$ bimodule $\E$ and a Hilbert $B$-$C$-bimodule $\F$, the $B$-balanced tensor product $\E \otimes_B \F$ is a Hilbert $A$-$C$-bimodule with left $A$-valued inner product satisfying 
\[
{}_{\E \otimes_B \F} \langle \xi_1 \otimes \eta_1 , \xi_2 \otimes \eta_2 \rangle = {}_{\E}\langle \xi_1  {}_\F\langle \eta_1, \eta_2\rangle, \xi_2 \rangle, \quad \xi_1, \xi_2 \in \E, \eta_1, \eta_2  \in \F,
\]
and right $C$-valued inner product satisfying 
\[
\langle \xi_1 \otimes \eta_1, \xi_2 \otimes \eta_2 \rangle_{\E \otimes_B \F} = \langle \eta_1, \langle \xi_1, \xi_2\rangle_\E  \eta_2 \rangle_{\F}, \quad \xi_1, \xi_2 \in \E, \eta_1, \eta_2 \in \F.
\]
If $\E$ is a Hilbert $A$-bimodule then we denote
\[ \E^{\otimes n} := \underbrace{\E \otimes_A \cdots \otimes_A \E}_{n \text{ times}}.\] 

\subsubsection{Right Hilbert $C(X)$-modules} 
When $A = C(X)$, we can construct Hilbert $A$-modules from vector bundles. For a good introduction to the theory of bundles, see \cite{Hus:fibre}.  Let $X$ be a compact metric space and $\mathscr{V} = [T, p, X]$ a Hermitian vector bundle of rank $n$. (Here $T$ denotes the total space of the bundle, $X$ the base space, and $p:T \to X$ the projection map.) We denote the set of continuous sections of $\mathscr{V}$ by $\Gamma(\mathscr{V} )$.  
Then $\Gamma(\mathscr V)$ has a natural right $C(X)$-module structure given by multiplication of a section by a continuous function. In fact,  it is a right Hilbert $C(X)$-module: Let $\mathcal{A}= \{(U_j,h_j)\}_{j=1}^N$ be an atlas of $\mathscr V$ and let $\gamma_1,\dots, \gamma_N$  a
partition of unity subordinate to $U_1,\dots, U_N$. Chart maps are always assumed to preserve inner products: $\langle h_j(x,v),h_j(x,w)\rangle_{p^{-1}(x)}=\langle v,w\rangle_{\mathbb C^n}$.} Then the $C(X)$-valued inner product 
\begin{equation} \label{innerproduct} \langle \xi,\eta\rangle_{\Gamma(\mathscr V)}(x):= \sum_{j=1}^N \gamma_j(x)\langle h_j^{-1}(\xi(x)),h_j^{-1}(\eta(x))\rangle_{\mathbb C^n}, \quad \xi, \eta \in \Gamma(\mathscr{V}),
\end{equation}
makes $\Gamma(\mathscr V)$ into a right Hilbert $C(X)$-module. Note that the inner product given in \eqref{innerproduct} does not depend on the choice of atlas nor the choice of a partition of unity. Moreover, with pairwise orthonormal vectors $s_{1},\dots, s_{n}$ of $\mathbb C^n$, the continuous sections given by
\[
\xi_{j,k} (x) := \begin{cases}
h_j(x, s_{k}\gamma_j(x)^{1/2}) , & x \in U_j  \\
0, & x \notin U_j.
\end{cases} 
\] generate $\Gamma(\mathscr V)$. Indeed, it is easily checked that 
\[\xi=\sum_{j=1}^N \sum_{k=1}^n \xi_{j,k}\langle \xi_{j,k},\xi\rangle_{\Gamma(\mathscr V)},\]  for every $\xi\in \Gamma(\mathscr V)$. 
Any continuous function $f  \in  C_0(U_j)$ allows us to define a section $\xi \in  \Gamma(\mathscr{V} )$ by setting
\[
\xi(x):=  \left\{ \begin{array}{ll} h_j(x, f(x)) & \text{if} \quad x \in U_j\\
0 & \text{otherwise}. \end{array} \right.
\]
By abuse of notation, when we refer to such a section $\xi(x)$ we will often write $h_j(x, f(x))$. 
If $\mathscr V$ is a line bundle, then the sections 
\begin{equation}\label{generator}
\eta_j(x):=h_j(x, \gamma_j(x)^{1/2})
\end{equation} 
generate $\Gamma(\mathscr V)$.

By the Serre--Swan theorem \cite[Theorem 2]{Swa:bundles} every finitely generated projective right $C(X)$-module is of the form $\Gamma(\mathscr{V})$ for a vector bundle over $X$. Furthermore, if $\E$ is  a right Hilbert $C(X)$-module, there is a unitary isomorphism $U : \E \to \Gamma(\mathscr{V})$, where $\Gamma(\mathscr{V})$ is equipped with an inner product as defined in \eqref{innerproduct} with respect to any choice of atlas for $\mathscr{V}$.


\subsubsection{Homeomorphisms twisted by line bundles}\label{theMainCharacter} We recall from \cite{AAFGJSV2024} that, when a Hilbert $C(X)$-bimodule is left and right full, it always arises from a homeomorphism twisted by a line bundle over $X$. 

Let $\alpha : X \to X$ be a homeomorphism and let $\mathscr{V} = [T, p, X]$ be a line bundle. 
We denote the \emph{pullback bundle} by $\alpha^*  \mathscr V$: the fibre $(\alpha^*\mathscr V)_x$ of $\alpha^*\mathscr V$ at $x\in X$ is  the fibre $\mathscr V_{\alpha(x)}$ of $\mathscr V$ at $\alpha(x)$.

We define the Hilbert $C(X)$-bimodule $\Gamma(\mathscr{V}, \alpha)$ as follows. As a right Hilbert $C(X)$-module, we have
\[ \Gamma(\mathscr{V}, \alpha)_{C(X)} \cong \Gamma(\mathscr{V}),\]
and we define the left action of $C(X)$ using $\alpha$ by setting
\[ f  \xi := \xi f \circ \alpha, \quad f \in C(X),\ \xi \in \Gamma(\mathscr{V}, \alpha). \]
The left inner product is then necessarily given by
\[ _{\Gamma(\mathscr{V}, \alpha)}\langle \xi_1, \xi_2 \rangle = \langle \xi_2, \xi_1 \rangle_{\Gamma(\mathscr{V})} \circ \alpha^{-1}.\]
Furthermore, as a left Hilbert $C(X)$-module, we have 
\[ _{C(X)}\Gamma(\mathscr{V}, \alpha) \cong \Gamma((\alpha^{-1})^*\mathscr{V}),\]
see \cite[Lemma 3.6]{AAFGJSV2024}. We refer to $\Gamma(\mathscr{V}, \alpha)$ as the Hilbert $C(X)$-bimodule obtained by \emph{twisting the homeomorphism $\alpha$ by the line bundle $\mathscr{V}$}. We refer to $\mathscr{V}$ as the \emph{twist}.

\begin{proposition}[{\cite[Propostion 3.7]{AAFGJSV2024}}] \label{prop.CharOfFull}
Let $\E$ be a non-zero Hilbert $C(X)$-bimodule. Suppose that $\E$ is finitely generated and projective as a right Hilbert $C(X)$-module. Then there exist a compact metric space $Y \subset X$, line bundle $\mathscr{V}= [T, p, X]$ and  homeomorphisms $\alpha : X \to Y$, $\beta : Y \to X$ such that $\E_{C(X)} \cong \Gamma(\mathscr{V})$, $_{C(X)}\E \cong \Gamma(\beta^*\mathscr{V})$ and $f \xi = \xi f \circ \alpha$ for every $f \in C(X)$ and every $\xi \in \E$. If $\E$ is left full, then we may take $X = Y$ and $\E \cong \Gamma(\mathscr{V}, \alpha)$.
\end{proposition}

\subsubsection{Orbit-breaking bimodules $\E_Y$}
The main objects of focus of this paper are the orbit-breaking subbimodules $\E_Y$  of  $\E=\Gamma(\mathscr V,\alpha)$, which we define below. 

\begin{definition}
Let $X$ be an infinite compact metric space, $\alpha : X \to X$ a homeomorphism and $\mathscr{V}$ a line bundle over $X$. Let $\E = \Gamma(\mathscr{V}, \alpha)$. For a non-empty closed subset $Y \subset X$,  \emph{the orbit-breaking bimodule} is defined  to be the subbimodule
\[ \E_Y := C_0(X \setminus Y) \E \]
that is obtained by restricting the left action on $\E$ to functions vanishing on $Y$. Note that $\E_Y$ is neither left nor right full as a Hilbert $C(X)$-bimodule.
\end{definition}

\subsection{Cuntz--Pimsner algebras}
Cuntz--Pimsner algebras were originally introduced by Pimsner \cite{Pimsner1997} and subsequently generalized by Katsura \cite{Katsura2003}. The class of Cuntz--Pimsner algebras includes all Cuntz--Krieger algebras, crossed products by $\mathbb{Z}$, crossed products by partial homeomorphisms and several other naturally occurring $\mathrm{C}^*$-algebras, see \cite[Section 3]{Katsura2003}.

 \begin{definition}[{\cite[Definition 2.1]{Katsura2004}}]  \label{def:cov rep} Let  $\mathcal{E}$ be a \mbox{$\mathrm{C}^*$-correspondence} over a $\mathrm{C}^*$-algebra $A$ with structure map $\varphi_\E : A \to \mathcal{L}(\mathcal{E})$. A \emph{representation} $(\pi, \tau)$ of $\E$ on a $\mathrm{C}^*$-algebra $B$ consists of a $^*$-homomorphism $\pi: A \to B$ and a linear map $\tau: \mathcal{E} \to B$ satisfying
 \begin{enumerate}
 \item $\pi(\langle \xi, \eta \rangle_\E) = \tau(\xi)^*\tau(\eta)$, for every $\xi, \eta  \in \E$;
 \item $\pi(a)\tau(\xi) = \tau(\varphi_\E(a)\xi)$, for every $\xi \in \E$, $a \in A$.
 \end{enumerate} 
Note that (1) and the $\mathrm{C}^*$-identity imply $\tau(\xi) \pi(a) = \tau(\xi a)$ for every $\xi \in \E$ and $a \in A$.
 
Let $\psi_{\tau}  : \mathcal{K}(\E) \to B$ be the $^*$-homomorphism defined on rank one operators by 
\[ \psi_{\tau}(\theta_{\xi, \eta}) = \tau(\xi)\tau(\eta)^*, \qquad \xi, \eta \in \E.\]
 We say that the representation $(\pi , \tau)$ is \emph{covariant} \cite[Definition 3.4]{Katsura2004} if in addition
\begin{enumerate}[resume]
\item $\pi(a) =  \psi_{\tau}(\varphi_\E(a))$, for every $a \in J_{\E}:=\varphi_\E^{-1}(\mathcal{K}(\E))\cap (\ker \varphi_\E)^{\perp}.$
\end{enumerate}
\end{definition}
In \cite[Section 4]{Katsura2004} Katsura constructs a \emph{universal covariant representation} $(\pi_u, \tau_u)$ satisfying the following: for any covariant representation $(\pi, \tau)$ of $\E$, there exists a surjective $^*$-homomorphism $\rho \colon \mathrm{C}^*(\pi_u, \tau_u) \rightarrow \mathrm{C}^*(\pi, \tau)$ such that $\pi=\rho \circ \pi_u$ and $\tau=\rho \circ \tau_u$, where ${\mathrm C}^*(\pi,\tau)$ is the  $\mathrm{C}^*$-subalgebra of $B$ generated by $\pi(A)$ and $\tau(\E)$.

\begin{definition}[{\cite[Definition 3.5]{Katsura2004}, see also \cite[Definition 2.4]{AEE:Cross}}] Let $A$ be a $\mathrm{C}^*$-algebra and let $\mathcal{E}$ be a \mbox{$\mathrm{C}^*$-correspondence} over $A$. The \emph{Cuntz--Pimsner algebra of $\mathcal{E}$ over $A$}, denoted $\mathcal{O}_A(\mathcal{E})$, is the $\mathrm{C}^*$-algebra $\mathrm{C}^*(\pi_u, \tau_u)$ generated by the universal covariant representation of $\E$. We write $\mathcal{O}(\E)$ if the \mbox{$\mathrm{C}^*$-algebra} $A$ is understood. 
\end{definition}

When $\E$ is a Hilbert $A$-bimodule, there is an equivalent construction due to Abadie, Eilers and Exel \cite{AEE:Cross}, which they call the crossed product of $A$ by the Hilbert $A$-bimodule $\E$.
By \cite[Proposition 3.7]{AAFGJSV2024} the $C(X)$-correspondences of the form $\Gamma(\mathscr{V}, \alpha)$ can be given the structure of a Hilbert bimodule only when $\mathscr{V}$ is a line bundle.

\begin{examples}\label{ex:putnam} Let $\alpha$ be a homeomorphism on an infinite compact metrizable space $X$, $\mathscr{V} = [T, p, X]$ a line bundle, and $\E:=\Gamma(\mathscr V,\alpha)$.
\begin{enumerate}
\item  If $\mathscr{V}$  is the trivial line bundle, then $\mathcal{O}(\E)$ is   isomorphic to the usual crossed product,  $\mathcal{O}(\E) \cong C(X) \rtimes_\alpha \mathbb{Z}$. 
\item At the other extreme, if $\alpha = \id_X$ is the trivial homeomorphism, then $\mathcal{O}(\E)$ is  a commutative $\mathrm{C}^*$-algebra, and actually $\mathcal{O}(\E)\cong C(P)$ where $P$ denotes the total space of the principal circle bundle $\mathcal{P}$ associated to the line bundle $\mathscr{V}$. (For further details see Section 7 in \cite{AAFGJSV2024}.)
\item Let  $Y \subset X$ be a non-empty closed subset of $X$ and $\E_Y: = C_0(X \setminus Y)\E$ be the orbit-breaking bimodule. If $\mathscr{V}$ is trivial, then $\mathcal{O}(\E_Y) \cong \mathrm{C}^*(C(X), C_0(X \setminus Y)u)\subset C(X) \rtimes_\alpha \mathbb{Z}$ is the orbit-breaking subalgebra introduced by Putnam in \cite{Putnam:MinHomCantor} and further studied in \cite{LinQPhil:KthoeryMinHoms, DPS:DynZ, DPS:OrbitBreaking} and elsewhere.
\end{enumerate}
\end{examples}

A representation $(\pi, \tau)$ of $\E$ is said to admit a \emph{gauge action} if, for each $z \in \mathbb{T}$, there exists a $^*$-homomorphism $\beta_z : \mathrm C^*(\pi, \tau) \to \mathrm C^*(\pi, \tau)$ such that $\beta_z(\pi(a)) = \pi(a)$ and $\beta_z(\tau(\xi)) = z \tau(\xi)$ for every $a \in A$ and every $\xi \in \E$ \cite[Definition 5.6]{Katsura2004}. The gauge-invariant uniqueness theorem for Cuntz--Pimsner algebras is stated as follows.

\begin{theorem}[{\cite[Theorem 6.4]{Katsura2004}}] \label{thm:gauge}
    For a covariant representation $(\pi,\tau)$ of a $\mathrm C^*$-correspondence $\E$ over a $\mathrm C^*$-algebra $A$, the surjective $^*$-homomorphism $\rho  : \mathcal{O}(\E) \to \mathrm C^*(\pi, \tau)$ is an isomorphism if and only if $(\pi, \tau)$ is injective (that is, $\pi$ is injective) and admits a gauge action.
\end{theorem}
 
The gauge action induces a topological grading 
\[ \mathcal{O}(\E) = \overline{\bigoplus_{n\in \mathbb Z} E_n}
\] where $E_n := \{ a \in \mathcal{O}_A(\E) \mid \beta_z(a) = z^n a\}.$
Evidently each $E_n$ is itself a Hilbert $A$-bimodule. For $n > 0$, we have $E_n\cong  \E^{\otimes n}$, while for $n=0$, $E_0 \cong A$.

To describe $E_n$ for $n < 0$, let
\[ \widehat{\E} = \{ \hat{\xi} \mid \xi \in \E\},\]
where $\hat\xi:\E\to A$ is an $A$-bilinear map defined by
$\hat{\xi}(\eta) = \langle \xi, \eta \rangle_\E$, $\xi, \eta \in \E.$
Then $\hat{\xi} \in \E^{\dagger}$, the dual Hilbert $A$-bimodule consisting of bounded $A$-linear maps $\phi : \E \to A$, with actions given by
\[ (a \phi  b)(\xi) = a b^* \phi(\xi), \qquad a, b \in A, \,\phi \in \E^{\dagger}.
\]
 When $\E$ is finitely generated projective as both a left and right $A$-module, we have that $\E^{\dagger} = \widehat{\E}$, see the discussion following Theorem 1.3 in \cite{Miscenko1979}. We make $\widehat{\E}$ into a Hilbert $A$-bimodule by defining
\[
{}_{\widehat{\E}}\langle \hat{\xi_1}, \hat{\xi_2} \rangle := \langle \xi_1, \xi_2 \rangle_\E, \qquad \langle \hat{\xi_1}, \hat{\xi_2} \rangle_{\widehat{\E}} := {}_\E\langle \xi_1, \xi_2 \rangle,
\]
for every $\xi_1, \xi_2 \in \E$. For $n <0$, we have $E_n \cong \widehat{\E}^{\otimes -n}$. 

Note that if  $\E^\dagger = \widehat{\E}$, then $\E$ is invertible in the sense that
\[ \widehat{\E} \otimes_A \E \cong \E \otimes_A \widehat{\E} \cong A.\]

\section{Tensor products of  Hilbert \texorpdfstring{$C(X)$-}{}bimodules \texorpdfstring{$\Gamma(\mathscr V, \alpha)$}{}} \label{sec:TP}

As we will use the tensor products of line bundles in the construction of the building blocks of the RSH decomposition of an orbit-breaking algebra, we discuss the relation between the tensor products of Hilbert $C(X)$-bimodules and the Hilbert $C(X)$-bimodule constructed from the tensor product of the underlying line bundles and the composition of homeomorphisms.

For a line bundle $\mathscr V=[T, p, X]$ and a homeomorphism $\alpha : X \to X$, denote by 
\[ \mathscr V^{(0)}:= X\times \mathbb C\]
the trivial line  bundle over $X$, and  for each $n \geq 1$, define a line bundle $\mathscr{V}^{(n)}$ by
\begin{align} \label{eq:LBtensor}
	\mathscr{V}^{(n)} &:=(\alpha^{n-1})^*\mathscr{V} \otimes (\alpha^{n-2})^*\mathscr{V} \otimes \cdots \otimes \alpha^*\mathcal{\mathscr{V}} \otimes  \mathscr{V}.
\end{align}
Whenever $\mathscr V|_{U_i}$, $\alpha^i(x)\in U_i$, is trivial via $h_{U_i}$ for all $i=0, \dots, n-1$,  we have a chart map 
\[(x,\lambda)\mapsto \lambda\, h_{U_{n-1}}(\alpha^{n-1}(x),1)\otimes\cdots\otimes h_{U_1}(\alpha(x),1)\otimes h_{U_0}(x,1)\] 
of the vector bundle $\mathscr V^{(n)}$ over $\cap_{i=0}^{n-1}\alpha^{-i}(U_i)$.

Note that  the fibre  $\mathscr V^{(n)}_x$ at $x\in X$ can be written as 

\begin{equation}\label{V^n}  
\mathscr{V}_x^{(n)} = \mathscr{V}_{\alpha^{m}(x)}^{(n-m)} \otimes \mathscr{V}_x^{(m)}
\end{equation}
whenever  $0\leq m\leq n$. 
In fact, 
\begin{align*} 	\mathscr{V}_x^{(n)} &=\mathscr{V}_{\alpha^{n-1}(x)} \otimes \mathscr{V}_{\alpha^{n-2}(x)} \otimes \cdots \otimes \mathscr{V}_{\alpha^{m}(x)} \otimes \mathscr{V}_{\alpha^{m-1}(x)} \otimes \cdots \otimes \mathscr{V}_{\alpha(x)}\otimes \mathscr{V}_{x} \\
	&=\mathscr{V}_{\alpha^{n-1-m}(\alpha^{m}(x))} \otimes\cdots \otimes \mathscr{V}_{\alpha^{m}(x)} \otimes \mathscr{V}_{\alpha^{m-1}(x)} \otimes \cdots \otimes \mathscr{V}_{\alpha(x)}\otimes  \mathscr{V}_{x} \\
	&= \mathscr{V}_{\alpha^{m}(x)}^{(n-m)} \otimes \mathscr{V}_x^{(m)}.
\end{align*}

The following result is surely well known when the homeomorphism $\alpha$ is trivial  and should be in the literature, however we could not find a reference. 

\begin{proposition}\label{tensortosection}
Let $\alpha:X\to X$ be a homeomorphism on a compact metric space $X$ and $\mathscr{V}$ be a line bundle over $X$. Then
there is a  Hilbert $C(X)$-bimodule  isomorphism
\[
\psi: \Gamma(\mathscr{V},\alpha)^{\otimes n} \to \Gamma( \mathscr{V}^{(n)} , \alpha^n)
\]
such that for $\xi_1 \otimes \xi_2 \otimes  \cdots \otimes \xi_n\in \Gamma(\mathscr{V},\alpha)^{\otimes n}$, we have
\[
 \psi(\xi_1 \otimes \xi_2 \otimes  \cdots \otimes \xi_n)(x) = \, \xi_1(\alpha^{n-1}(x)) \otimes \xi_2(\alpha^{n-2}(x)) \otimes  \cdots \otimes  \xi_n(x).
\]
If $n=0$, then  by $\psi$ we mean the identity map on $C(X)$. 
\end{proposition}

\begin{proof} If $\xi_i\in \Gamma(\mathscr V)$, $i=1,\dots, n$, then clearly the map 
\[x\mapsto \xi_1(\alpha^{n-1}(x))\otimes\cdots \otimes \xi_{n-1}(\alpha(x))\otimes\xi_n(x), \ x\in X,\] is a continuous section in $\Gamma(\mathscr V^{(n)})$, and thus we have a multilinear map sending $(\xi_1, \dots, \xi_n)\in \Gamma(\mathscr V)\times\cdots\times \Gamma(\mathscr V)$ to this section in $\Gamma(\mathscr V^{(n)})$. Then there is a linear map $\psi:\Gamma(\mathscr V)^{\otimes n}\to \Gamma(\mathscr V^{(n)})$ satisfying the identity in the proposition. For $f\in C(X)$, 
\[ \psi(\xi_1 \otimes \xi_2 \otimes  \cdots \otimes \xi_n f)(x) = (\psi(\xi_1 \otimes \xi_2 \otimes  \cdots \otimes \xi_n )f)(x)\] is rather obvious, and also we have
\begin{align*}
\lefteqn{\psi(f\cdot\xi_1 \otimes \xi_2 \otimes  \cdots \otimes \xi_n)(x)} \\ 
&=\psi(\xi_1 \otimes \xi_2 \otimes  \cdots \otimes \xi_n f\circ\alpha^n)(x)\\
&=\xi_1(\alpha^{n-1}(x))\otimes \cdots \otimes \xi_{n-1}(\alpha(x))\otimes \xi_n(x) f\circ\alpha^n(x)\\
&=  f\cdot\psi(\xi_1 \otimes \xi_2 \otimes  \cdots \otimes \xi_n)(x).       
\end{align*}  Thus $\psi$ is a $C(X)$-bimodule map from $\Gamma(\mathscr V,\alpha)^{\otimes n}$ to $\Gamma(\mathscr V^{(n)},\alpha^n)$.   

To show that $\psi$ preserves the inner product, with $\E:=\Gamma(\mathscr V,\alpha)$, we calculate 
\begin{align*}
\lefteqn{\langle \xi_1 \otimes   \cdots \otimes \xi_n, \eta_1 \otimes  \cdots \otimes \eta_n \rangle_{\E^{\otimes n}} (x)}\\
&= \langle \xi_2 \otimes  \cdots \otimes \xi_n,\, \langle \xi_1, \eta_1 \rangle_\E \eta_2 \otimes  \cdots \otimes \eta_n \rangle_{\E^{\otimes n-1}}(x)\\
&= \langle \xi_2 \otimes  \cdots \otimes \xi_n,  \eta_2 \otimes  \cdots \otimes \eta_n\rangle_{\E^{\otimes n-1}}(x)\langle \xi_1, \eta_1 \rangle_\E \circ \alpha^{n-1}(x)\\
&=\quad \cdots\\
&= \langle \xi_n,  \eta_n \rangle_\E (x)\langle \xi_{n-1}, \eta_{n-1} \rangle_\E \circ \alpha(x)\cdots\langle \xi_1, \eta_1 \rangle \circ \alpha^{n-1}(x),
\end{align*}
and
\begin{align*}
\lefteqn{\langle \psi(\xi_1 \otimes \xi_2 \otimes  \cdots \otimes \xi_n), \psi( \eta_1 \otimes \eta_2 \otimes  \cdots \otimes \eta_n) \rangle_{\Gamma(\mathscr V^{(n)})} (x)} \\
&= \langle\xi_1(\alpha^{n-1}(x)) \otimes \cdots\otimes \xi_n(x), \eta_1(\alpha^{n-1}(x)) \otimes   \cdots \otimes \eta_n(x) \rangle_{\mathscr V^{(n)}_x}\\
&= \langle\xi_1(\alpha^{n-1}(x)), \eta_1(\alpha^{n-1}(x)) \rangle_{\mathscr V_{\alpha^{n-1}(x)}} \cdots \langle \xi_n(x),   \eta_n(x) \rangle_{\mathscr V_x},
\end{align*}
and compare them.  
We see that these two match  up to the order of multiplication since \[\langle \xi_i,\eta_i\rangle_\E (\alpha^{n-i}(x))=\langle \xi_i(\alpha^{n-i}(x)),\eta_i (\alpha^{n-i}(x))\rangle_{\mathscr V_{\alpha^{n-i}(x)}}.\]
This inner product preserving property implies that $\psi$ is injective.

Let $\eta_j$ be the generators of $\E=\Gamma(\mathscr V)$ as in (\ref{generator}). From the definition of $\mathscr V^{(n)}$,  one can show that the sections $\psi(\eta_{i_1}\otimes \eta_{i_2}\otimes \cdots\otimes \eta_{i_n})$  generate $\Gamma( \mathscr{V}^{(n)} , \alpha^n)$. Then
for $\xi \in \Gamma(\mathscr{V}^{(n)},\alpha^n)$,
the element 
 \begin{align*}
\mu &:=\sum_{i_1,i_2,\ldots,i_n}    \eta_{i_1}\otimes \eta_{i_2}\otimes \cdots\otimes \eta_{i_n} \langle \psi(\eta_{i_1}\otimes \eta_{i_2}\otimes \cdots\otimes \eta_{i_n}), \xi\rangle_{\Gamma(\mathscr V^{(n)})}	
\end{align*} in $\E^{\otimes n}$ 
satisfies $\psi(\mu)=\xi$.  
Thus $\psi$ is surjective.
\end{proof}

We end this section with the following proposition about orbit-breaking bimodules which will be used in the following sections.

\begin{proposition}[see {\cite[Lemma 6.2]{AAFGJSV2024}}]\label{EandSections}
Let $\alpha:X\to X$ be a homeomorphism and $\mathscr{V}$ be a line bundle on $X$. Let $\psi:\mathcal E^{\otimes m}\to \Gamma(\mathscr V^{(m)},\alpha^m)$ be the isomorphism given by Proposition \ref{tensortosection} with $\E=\Gamma(\mathscr V,\alpha)$. Let $Y \subset X$ be a non-empty closed subset and denote by $\E_Y$ the associated orbit-breaking bimodule. Then for any element $\xi \in \mathcal{E}^{\otimes m}$, we have that $\xi \in \mathcal{E}_Y^{\otimes m}:=(\E_Y)^{\otimes m}$ if and only if  $\psi(\xi)$  vanishes on $\alpha^{-1}(Y)\cup \alpha^{-2}(Y)\cup \cdots\cup \alpha^{-m}(Y)$. If $m=0$, then this set is empty and by $\psi$ we mean the identity map of $C(X)$.
\end{proposition}

\begin{proof} 
If $\xi=  f_1\eta_1 \otimes \cdots \otimes f_m\eta_m $ for some $\eta_i \in \mathcal{E}$, $f_i\in C_0(X\setminus Y)$, $i=1,2,\ldots,m$, then 
\begin{align*}
\xi  &= f_1\eta_1 \otimes \cdots \otimes f_m\eta_m   \\
& =  \eta_1 \otimes \cdots \otimes  \eta_m (f_{1}\circ \alpha^m)\cdots(f_{m-1}\circ \alpha^2)( f_m\circ \alpha).
\end{align*}
Thus $\psi(\xi)(x)=0$ whenever   $x \in \alpha^{-1}(Y)\cup  \cdots\cup \alpha^{-m}(Y)$. 
Since such $\xi$ generate  $\mathcal{E}_Y^{\otimes m}$, the ``only if'' part follows. 
 
Conversely, suppose  $\xi \in \mathcal{E}^{\otimes m} $ satisfies  $\psi(\xi)(x) = 0$ for every $x \in \alpha^{-1}(Y)\cup \alpha^{-2}(Y)\cup \cdots\cup \alpha^{-m}(Y)$. 
Let $\{h_j :U_j \times \mathbb{C} \to \mathscr{V}|_{U_j}\}_{j=1,\ldots,n}$ be an atlas for $\mathscr{V}$ and $\gamma_1,\ldots,\gamma_n$ a
partition of unity subordinate to $U_1,\ldots,U_n$. Then $\eta_j$, $1 \leq j \leq n$, given by $\eta_j(x)= h_j(x,{\gamma_j}^{1/2}(x))$,  generate $\E$ as a right Hilbert $C(X)$-module, and so  $\eta_{i_1}\otimes  \cdots\otimes \eta_{i_m}$, $1\leq i_1, \dots,  i_m\leq n$, generates $\E^{\otimes m}$. 
 The functions 
 \[
 f_{i_1,i_2,\ldots,i_m} (x) := \langle \psi(\eta_{i_1}\otimes \cdots\otimes \eta_{i_m}), \psi(\xi)\rangle_{\Gamma(\mathscr V^{(m)})} (x)= \langle \eta_{i_1}\otimes \cdots\otimes \eta_{i_m}, \xi\rangle_{\E^{\otimes m}} (x)
 \]
 vanish on $\alpha^{-1}(Y)\cup \alpha^{-2}(Y)\cup \cdots\cup \alpha^{-m}(Y)$. 
 Since  the functions $f_{i_1,i_2,\ldots,i_m}$  are linear combinations of positive ones, we may assume that they are positive. Then we have
 \begin{align*}
\xi &=\sum_{1\leq i_1, \dots,i_m\leq n}     \eta_{i_1}\otimes\cdots\otimes \eta_{i_m} \langle \eta_{i_1}\otimes \cdots\otimes \eta_{i_m}, \xi\rangle_{\E^{\otimes m}} \\
& =\sum_{1\leq i_1, \ldots,i_m\leq n}   \eta_{i_1}\otimes \cdots\otimes \eta_{i_m} f_{i_1,i_2,\ldots,i_m} \\
& =\sum_{1\leq i_1, \ldots,i_m\leq n}  \big[(f_{i_1,i_2,\ldots,i_m}\circ \alpha^{-m})^{1/m} \eta_{i_1}\otimes (f_{i_1,i_2,\ldots,i_m}\circ \alpha^{-m+1})^{1/m} \eta_{i_2}\otimes \cdots \\
 &\qquad\qquad\qquad\qquad\qquad\qquad \qquad\qquad \cdots\otimes (f_{i_1,i_2,\ldots,i_m}\circ \alpha^{-1})^{1/m} \eta_{i_m}\big].	
\end{align*}
Thus  $\xi \in  \mathcal{E}_Y^{\otimes m} $.
\end{proof}

\section{Orbit-breaking algebras} \label{sec:OBA}

Generalizing Example~\ref{ex:putnam}.(3), we define orbit-breaking algebras via orbit-breaking bimodules. The definition of orbit-breaking algebras in this context first appeared in \cite{AAFGJSV2024}.

\begin{definition}
    Let $\alpha : X \to X$ be a homeomorphism on an infinite compact metric space $X$ and $\mathscr{V}$ be a line bundle over $X$. Let $\E=\Gamma(\mathscr{V},\alpha)$. Suppose that $Y \subseteq X$ is a non-empty closed subset of $X$. The \emph{orbit-breaking algebra of $\mathcal{O}(\E)$} at $Y$ is the Cuntz--Pimsner algebra $\mathcal{O}(\E_Y)$ of the orbit breaking submodule $\E_Y=C_0(X\setminus Y)\E$.
\end{definition}

Note that the orbit-breaking algebra $\mathcal{O}(\E_Y)$ can be identified with the $\mathrm C^*$-subalgebra $\mathrm C^*(C(X), C_0(X \setminus Y) \E)$ of $\mathcal{O}(\E)$. 

Let $A$ be a $\mathrm C^*$-algebra and $\E$ a $\mathrm{C}^*$-correspondence over $A$ with the structure map $\varphi_\E$. 
We say that $\E$ is \emph{minimal} if, for any non-zero ideal $J \subset A$, if $\langle \E, \varphi_\E(J) \E \rangle_\E \subset J$, then $J = A$.  This definition is a generalization of the corresponding definition for a single homeomorphism $\alpha : X \to X$  on a compact metric space: We say that $\alpha$ is \emph{minimal} if, for every non-empty closed subset $Y \subset X$, if $\alpha(Y) \subset Y$, then $Y = X$.

\begin{theorem}
    Let $X$ be an infinite compact metric space, $\alpha : X \to X$ a homeomorphism, and $\mathscr{V}$ a line bundle over $X$. Let $Y \subset X$ be a non-empty closed subset. Then
    \begin{enumerate}
        \item $\Gamma(\mathscr{V}, \alpha)$ is minimal if and only if $\alpha$ is minimal if and only if the crossed product $\mathcal{O}(\Gamma(\mathscr{V}, \alpha)) $ is simple. 
        \item If $\alpha$ is minimal, then $\mathcal{O}(C_0(X \setminus Y) \Gamma(\mathscr{V}, \alpha))$ is simple if and only if $Y \cap \alpha^n(Y) = \emptyset$ for every $n \in \mathbb{Z} \setminus \{0\}$.
        \item Both $\mathcal{O}(\Gamma(\mathscr{V}, \alpha))$ and $\mathcal{O}(C_0(X \setminus Y) \Gamma(\mathscr{V}, \alpha))$ are nuclear and satisfy the UCT.
    \end{enumerate}
    \end{theorem}

    \begin{proof}
That $\Gamma(\mathscr{V}, \alpha)$ is minimal if and only if $\alpha$ is minimal if and only if the crossed product $\mathcal{O}(\Gamma(\mathscr{V}, \alpha)) $ is simple follows from \cite[Proposition 3.9, Corollary 3.11]{AAFGJSV2024}. Suppose that $\alpha$ is minimal. That $\mathcal{O}(C_0(X \setminus Y) \Gamma(\mathscr{V}, \alpha))$ is simple if and only if $Y \cap \alpha^n(Y) \neq \emptyset$  for $n \neq 0$ follows from \cite{Kettner2024-2}. (Note that this is a slight improvement of the result stated in \cite{AAFGJSV2024}). Applying \cite[Theorem 7.3, Proposition 8.8]{Katsura2004} to Cuntz--Pimsner algebras $\mathcal{O}(\Gamma(\mathscr{V},\alpha))$  and $\mathcal{O}(C_0(X \setminus Y) \Gamma(\mathscr{V}, \alpha))$ yields (3).
    \end{proof}


\section{Recursive subhomogeneous algebras} \label{sec:rsh} 
In this section, we begin by recalling the definition of and terminology for recursive subhomogeneous algebras. 

Next, we recall the recursive subhomogeneous algebra structure of orbit-breaking subalgebras for crossed products by minimal homeomorphisms. This is a special case of our construction.

We then provide the conceptual steps for our more general orbit-breaking algebras,  $\mathcal O(C_0(X \setminus Y) \Gamma(\mathscr{V}, \alpha))$ for $X$ an infinite compact metric space, $\alpha : X \to X$ a minimal homeomorphism, $\mathscr{V}$ a line bundle over $X$ and $Y \subset X$ a closed subset with non-empty interior. The aim is to highlight the similarities and differences to the construction in the crossed product case, as well as to aid the reader in understanding the technical constructions in subsequent sections.
 
\subsection{Recursive subhomogeneous algebras}
Let $A$, $B$ and $C$ be $\mathrm{C}^*$-algebras and $\varphi : A \to C$ and $\rho : B \to C$ be homomorphisms. The \emph{pullback} $A \oplus_{C} B$ is defined to be the algebra of all $(a,b) \in A \oplus B$ such that $\varphi(a)=\rho(b)$.
\begin{definition}[\cite{Phillips:recsub}]
    A \emph{recursive subhomogeneous algebra} is a $\mathrm{C}^*$-algebra given by the following recursive definition.
   \begin{enumerate}
        \item If $X$ is a compact Hausdorff space and $n \geq 1$, then $C(X, M_n)$ is a recursive subhomogeneous algebra.
        \item If $A$ is a recursive subhomogeneous algebra, $X$ is a compact Hausdorff space, $X^{(0)} \subseteq X$ is closed, $\varphi : A \to C(X^{(0)}, M_n)$ is any unital homomorphism, and $\rho : C(X, M_n) \to C(X^{(0)}, M_n)$ is the restriction homomorphism, then the pullback $A \oplus_{C(X^{(0)}, M_n)}C(X, M_n)$ is a recursive subhomogenous algebra.
\end{enumerate}
\end{definition}
 Observe that any recursive subhomogeneous $\mathrm{C}^*$-algebra $R$ can be written in the form
\begin{equation} \label{rsh-decomp} R \cong\left[ \cdots \left[ \left[ C_0 \oplus_{C^{(0)}_1} C_1 \right] \oplus_{C_2^{(0)}} \right] \cdots \right] \oplus_{C_l^{(0)}} C_l ,
\end{equation}
 where $C_j = C(X_j, M_{n(j)})$ for compact Hausdorff spaces $X_j$, $n(j) \in \mathbb{Z}_{>0}$ and $C^{(0)}_j = C(X_j^{(0)}, M_{n(j)})$ for (possibly empty) compact subsets $X^{(0)}_j \subset X_j$, and where the maps $C_j \to C_j^{(0)}$ are the restriction maps. The expression \eqref{rsh-decomp} is called a \emph{recursive subhomogeneous decomposition} of $R$. A recursive subhomogeneous decomposition for a given $\mathrm{C}^*$-algebra is in general not unique.

\begin{definition}
Let \[ R = \left[ \cdots \left[ \left[ C_0 \oplus_{C^{(0)}_1} C_1 \right] \oplus_{C_2^{(0)}} \right] \cdots \right] \oplus_{C_l^{(0)}} C_l \]  be a recursive subhomogeneous decomposition. Then 
\begin{enumerate}  
\item $l$ is called the \emph{length} of the decomposition, 
\item the recursive subhomogeneous algebra    \[R^{(k)} = \left[ \cdots \left[ \left[ C_0 \oplus_{C^{(0)}_1} C_1 \right] \oplus_{C_2^{(0)}} \right] \cdots \right] \oplus_{C_k^{(0)}} C_k, \]     is called the \emph{$k$\textsuperscript{th} stage algebra},     
\item $X_0, \dots, X_l$ are called the \emph{base spaces} of the decomposition and their disjoint union $\bigsqcup_{j=0}^l X_j$ is called the \emph{total space},    
\item $n(0), \dots, n(j)$ are called the \emph{matrix sizes}.
\end{enumerate}
\end{definition}

\subsection{Orbit-breaking in crossed products by homeomorphisms}

Let $\alpha$ be a minimal homeomorphism on an infinite compact metrizable space $X$ and let $Y$ be a  closed subset of $X$ with non-empty interior. 
For $x \in Y$, the \emph{first return time} $r_Y(x)$ of $x$ to $Y$ is the smallest integer $n \geq 1$ such that $\alpha^n(x) \in Y$. 
Since $\alpha$ is minimal and $Y$ has non-empty interior, $\sup_{x \in Y} r_Y(x) <\infty$. A compactness argument shows that there are only finitely many distinct return times to $Y$, which we denote by $r_1 < \dots < r_K$. For  $1 \leq k \leq K$, set
\begin{equation}\label{Y-k}
  Y_k:=\{x \in Y \mid r_Y(x)=r_k\}. 
\end{equation}
For every $k=1, \dots, K$, the union $\cup_{i=1}^k Y_i$ is closed. It is easy to check that 
\begin{equation}\label{X:union}
Y=\bigsqcup_{k=1}^K Y_k, \quad \text{and} \quad X= \bigsqcup_{k=1}^K \bigsqcup_{i=0}^{r_k-1}\alpha^i(Y_k).    
\end{equation} 
The sequences \[Y_k,\alpha(Y_k),\, \dots, \,\alpha^{r_k-1}(Y_k)\] are called \emph{Rokhlin towers}. The sets $Y_k$ are the \emph{bases} of towers and $r_k$ are their \emph{heights}.

Denote by $A_Y$ the orbit-breaking subalgebra $\mathrm{C}^*(C(X), C_0(X \setminus Y)u) \subset C(X)\rtimes_\alpha \mathbb{Z}$. Here $u$ denotes the unitary implementing $\alpha$ in the crossed product.

\subsubsection{Orbit-breaking for minimal Cantor systems}
Suppose that $X$ is the Cantor space. In \cite{Putnam:MinHomCantor}, Putnam constructs Rokhlin towers for $Y$ consisting of clopen sets. Since the sets are clopen, their indicator functions are continuous and generate projections in $A_Y$. Since $\alpha^j(Y_k) \cap Y = \emptyset$ for every $1 \leq j \leq r_k-1$ we have that $\chi_{\alpha^j(Y_k)} u \in A_Y$ for every $1 \leq j \leq r_k-1$. Setting $e_{j, j-1} := \chi_{\alpha^j(Y_k)} u$, it is easy to check that $e_{j, j-1}$ generate a copy of $M_{r_k}$. This allows Putnam to show that $A_Y$ is isomorphic with a direct sum of homogeneous algebras $M_{r_k}(Y_k)$~\cite[Theorem 3.3]{Putnam:MinHomCantor}.  

\begin{example}[{\cite[Theorem 3.3]{Putnam:MinHomCantor}, see also~\cite[Lemma 11.2.22]{GioKerPhi:CRM}}]
    Let $X$ be a Cantor set, $Y$ be a non-empty compact open subset and $\alpha : X \to X$ be a minimal homeomorphism. Then each $Y_k$ is clopen and compact. Set $p_k :=\chi_{Y_k}$. Then each $p_k$ commutes with elements in $A_Y$  and we have $A_Y=\bigoplus_{k=1}^Kp_k A_Y p_k$. Then $\pi_k \colon p_k A_Y p_k \to C(Y_k, M_{r_k})$ given by \[\pi_k(f)=\diag(f|_{Y_k}, \dots, f\circ \alpha^{r_k-1}|_{Y_k}),\]
for $f \in C(X)$ and 
\[
\pi_k(\chi_{\alpha^j(Y_k)}u)=e_{j,j-1} \otimes 1
\]
is an isomorphism and we have $A_Y \cong \bigoplus_{k=1}^K C(Y_k, M_{r_k})$.
\end{example}

\subsubsection{Orbit-breaking subalgebras of arbitrary crossed products}

If $X$ is not a Cantor space and $Y$ is not clopen, then the sets $Y_k$'s are not necessarily closed. If $x \in \overline{Y_k} \setminus Y_k$, then by continuity, $\alpha^{r_k}(y) \in Y$. However, this might not be the \emph{first} return time of $y$ to $Y$. In particular, for any $1 < k \leq K$, the intersection $\overline{Y_k} \cup Y_j$ for $j < k$ may be nonempty. In this case, we can only expect a recursive subhomogeneous decomposition for $A_Y$, rather than a direct sum of homogeneous algebras. 

\begin{example}[{\cite{QLin:Ay}, see also~\cite[Theorem 11.3.19]{GioKerPhi:CRM}}] Let $X$ be a compact metric space, $\alpha : X \to X$ a minimal homeomorphism, and $Y$ a closed subset of $X$ with non-empty interior.
Set 
\[
B_Y=\bigoplus_{k=1}^K C(\overline{Y_k}, M_{r_k}).
\]
There exists an embedding 
\[ \pi : A_Y \to B_Y.\] 
 Let $\pi_k : A_Y \to C(\overline{Y_k}, M_{r_k})$ be the composition of $\pi$ with the projection onto the $k$\textsuperscript{th} summand of $B_Y$. Set $\pi^{(k)}=\oplus_{i=1}^k \pi_k$ and $B_k=\pi^{(k)}(A_Y)$. 
 
 We say that an element $b=(b_1, \dots,b_K)$ has the \emph{boundary decomposition property} if it satisfies the following: for any $x\in \overline{Y_k}\setminus Y_k$, $k=1,\dots,K$, with 
\[
x \in  \overline{Y_k} \cap Y_{t_1} \cap \alpha^{-{r_{t_1}}}(Y_{t_2}) \cap \cdots
\cap \alpha^{-(r_{t_1} + r_{t_2} + \cdots + r_{t_{m-1}})}(Y_{t_{m}})
\] 
and $r_{t_1} + r_{t_2} + \cdots + r_{t_m} = r_{ k}$, then $b_k$ is given by the block diagonal matrix
\[
b_k(x)=\diag(b_{t_1}(x), b_{t_2}(\alpha^{r_{t_1}}(x)), \dots, b_{t_s}(\alpha^{r_{t_1}+\dots+r_{t{m-1}}}(x))).
\]
An element $b=(b_1, \dots, b_K) \in B_Y$ is in $\pi(A_Y)$ if and only if $b$ has the boundary decomposition property. This result enables us to explicitly describe the recursive subhomogeneous construction of the orbit-breaking algebra $A_Y$.

The $1$\textsuperscript{st} stage algebra is given by $B_1 = C(\overline{Y_1}, M_{r_1})$. (In fact, $Y_1$ is already a closed subset.)

Fix $1 \leq k \leq K$ and let $B_{k-1}$ denote the $(k-1)$\textsuperscript{th} stage algebra. To define the $k$\textsuperscript{th} stage algebra $B_k$, we construct the pullback 
\[\xymatrix{
    B_{k} \ar[rr] \ar[d]& & C(\overline{Y_k}, M_{r_k}) \ar[d]^{\rho_{k}} \\
    B_{k-1} \ar[rr]_-{\varphi_{k}} & & C(\overline{Y_k}\setminus Y_k),}
\]
where $\rho_k$ is restriction, and 
$\varphi_k : B_{k-1} \to C(\overline{Y_k} \setminus Y_k, M_{r_k})$ is defined as follows: Let $(b_1, \dots, b_{k-1}) \in B_{k-1}$. For $k \leq j \leq K$, there are  $b_j \in C(\overline{Y_j}, M_{r_j})$ such that $b=(b_1, \dots, b_{k-1},b_k, \dots, b_K) \in \pi(A_Y)$. Since $b$ satisfies the boundary decomposition property, $b_k|_{ \overline{Y_k}\setminus Y_k}$ only depends on $(b_1, \dots, b_{k-1})$. Now, define $\varphi_k(b_1, \dots, b_{k-1})(x)=b_k(x)$. Then 
\[
B_k=B_{k-1} \oplus_{C(\overline{Y_k} \setminus Y_k, M_{r_k})} C(\overline{Y_k}, M_{r_k}).
\]
It follows that $A_Y$ has a recursive subhomogeneous decomposition given by 
\begin{equation}\label{rsh-trivial}
[\cdots[C(Y_1,M_{r_1})\oplus_{C( \overline{Y_2}\setminus Y_2, M_{r_2})}C(\overline{Y}_2, M_{r_2})] \dots \oplus_{C( \overline{Y_K}\setminus Y_K, M_{r_K})} C(\overline{Y_K}, M_{r_K})]. 
\end{equation}
\end{example}

\subsection{Generalizing to a non-trivial line bundle} Now assume further that there is a line bundle $\mathscr{V}$ over $X$ and set $\E:=\Gamma(\mathscr{V}, \alpha)$. The main goal of this paper is to show that, for $Y \subset X$ a closed subset with non-empty interior, the orbit-breaking algebra $\mathcal{O}(\E_Y)$ is a recursive subhomogeneous algebra. Here we outline the steps involved in the construction.
 
\textbf{Step 1.} In Section~\ref{sec:endomorphism bdl}, we introduce the building blocks in the recursive subhomogeneous decomposition of the orbit-breaking algebra $\mathcal{O}(\E_Y)$.  Let $\mathscr{D}^{(r_k)} := \mathscr{V}^{(0)} \oplus \dots  \oplus \mathscr{V}^{(r_k-1)}$ denote the rank $r_k$ vector bundle over $X$ given by the direct sum of the line bundles $\mathscr{V}^{(i)}$ defined in \eqref{eq:LBtensor}, and let $\mathscr M_k :=End(\mathscr{D}^{(r_k)})|_{\overline{Y}_k}$ be the endomorphism bundle (see Section~\ref{sec:endomorphism bdl}). When the twist $\mathscr{V}$ is no longer trivial, we replace $C(\overline{Y}_k, M_{r_k})$ in~\eqref{rsh-trivial} by $\Gamma(\mathscr{M}_k)$, see Theorem~\ref{thm:rsh}.\\

\textbf{Step 2.} In Section~\ref{sec:embedding}, we construct an embedding from the orbit-breaking algebra $\mathcal{O}(\E_Y)$ into $\Gamma(\mathscr{M}_1) \oplus \dots \oplus \Gamma(\mathscr{M}_K)$. For this, we introduce a covariant representation $(\pi_k, \tau_k)$ of the Hilbert $C(X)$-bimodule $\E_Y$ to $\Gamma(\mathscr{M}_k)$ which induces a $*$-homomorphism, again denoted by $\pi_k$, from the orbit-breaking algebra to $\Gamma(\mathscr{M}_k)$, for all $1 \leq k \leq K$. In Proposition~\ref{prop:gauge action} we show that $\pi=\oplus_{k=1}^K\pi_k$ is injective and its image admits a gauge action, which, by the gauge invariant uniqueness theorem, implies that $\pi$ is a $^*$-isomorphism from $\mathcal{O}(\E_Y)$ onto its image.\\

\textbf{Step 3.} To describe the pullbacks in the RSH structure, we need to understand the image of the $*$-representation $\pi$. We generalize the boundary decomposition property for elements in $\Gamma(\mathscr{M}_1) \oplus \dots \oplus \Gamma(\mathscr{M}_K)$ in Definition~\ref{bd property}. Then by Theorem~\ref{thm:bdp} the range of $\pi$ is precisely the set of all elements with boundary decomposition property.\\

\textbf{Step 4.} Finally we obtain the main result of the paper. The boundary decomposition property of elements in the image of $\pi$ enables us to define $*$-homomorphisms 
\[\varphi_k :(\oplus_{i=1}^{k} \pi_i) (\mathcal{O}(\E_Y))\to \Gamma(\mathscr M_{k+1}|_{\overline{Y_{k+1}}\setminus Y_{k+1}}),\] 
$k=1, \dots, K-1$. Then by Theorem~\ref{thm:rsh}, the pullback 
\[(\oplus_{i=1}^{k} \pi_i)(\mathcal{O}(\E_Y)) \oplus_{\Gamma(\mathscr{M}_{k+1}|_{\overline{Y_{k+1}} \setminus Y_{k+1}})}\Gamma(\mathscr{M}_k)\] via $\varphi_k$ and the restriction map $\rho_k: \Gamma(\mathscr M_{k+1})\to \Gamma(\mathscr M_{k+1}|_{\overline{Y_{k+1}}\setminus Y_{k+1}})$ is equal to $(\oplus_{i=1}^{k+1} \pi_i)(\mathcal{O}(\E_Y))$.
That is,
\begin{align*}
  \mathcal{O}(\E_Y)\cong &\ \pi(\mathcal{O}(\E_Y))\\
  = &\  [\cdots [\Gamma(\mathscr{M}_1)\oplus_{\Gamma(\mathscr{M}_2|_{\overline{Y_2} \setminus Y_2})} \Gamma(\mathscr{M}_2)]\oplus \cdots ]\oplus_{\Gamma(\mathscr{M}_K|_{\overline{Y_K} \setminus Y_K})} \Gamma(\mathscr{M}_K).  
\end{align*}
Since the section algebra of each matrix bundles has a recursive subhomogeneous decomposition \cite[Proposition 1.7]{Phillips:recsub}, we conclude that the iterated pullback describing the orbit-breaking algebra also has an RSH decomposition.



\section{Endomorphism bundles}\label{sec:endomorphism bdl}
In this section we construct endomorphism bundles  of vector bundles constructed from line bundles. The section algebras of these endomorphism bundles will be the building blocks in the recursive subhomogeneous decomposition of the orbit-breaking algebra. We discuss the local trivialization of these bundles in detail.

First we recall that for a vector bundle $\mathscr W$ of rank $n$ over $X$, the {\it endomorphism bundle} $End(\mathscr W)$ is a vector bundle  over $X$ whose fibre at $x$ consists of all  linear maps on the fibre vector space $\mathscr W_x$ of $\mathscr W$, namely $End(\mathscr W)_x=End(\mathscr W_x)$ for every $x$ (see \cite[6.8 Example]{Hus:fibre}).

 Let $\mathcal U$ be an open cover of $X$ and let $\{(U, h_U : U \times \mathbb{C}^n\to \mathscr W|_U)\}_{U\in \mathcal U}$ be an atlas of $\mathscr W$ with transition functions $\{g_{U,V}\}_{U,V \in \mathcal{U}}$. Then $\mathcal{U}$ induces an atlas $\{(U, H_U : U \times  M_n(\mathbb{C})  \to End(\mathscr W)|_U)\}$ of $End(\mathscr W)$ where the $H_U$ are defined by
\begin{equation}\label{chart-endobundle} H_U(x, M):=h_U(x,\cdot)\circ M\circ h_U(x, \cdot)^{-1}\in End(\mathscr W_x),\end{equation} 
for $M\in M_n(\mathbb C)$. This is represented in the following diagram  with $h_{U,x}:=h_U(x,\cdot)$:
\begin{center}
\begin{tikzpicture}[align=center,node distance=3cm]
  \node (A) {$\mathbb{C}^n$};
  \node (B) [below= 1cm of A] {$\mathscr W_x$} ;
  \node (C) [right=of A] {$\mathbb{C}^n$};
  \node (D) [right=of B] {$\mathscr W_x$.};
  \draw[-stealth] (A)-- node[left] {\small $h_{U,x}$} (B);
  \draw[-stealth] (B)-- node [below] {\small $H_U(x,M)$} (D);
  \draw[-stealth] (A)-- node [above] {\small $M$} (C);
  \draw[-stealth] (C)-- node [right] {\small $h_{U,x}$} (D);
\end{tikzpicture}
\end{center}
 Moreover,   whenever $U, V \in \U$ satisfy $U \cap V \neq \emptyset$, the map
\[
H_V^{-1} \circ H_U : (U\cap V)\times M_n(\mathbb{C})\to (U\cap V)\times M_n(\mathbb{C})
\]
is well defined and satisfies
\begin{align*}
H_V^{-1} \circ H_U (x,M) &= \left (x, (h_{V, x})^{-1}\circ  h_{U, x} \circ M \circ (h_{U, x})^{-1} \circ h_{V, x}  \right)\\
&= \left (x, \mathrm{Ad}(g_{V,U}(x)) M \right),
\end{align*} 
where $\mathrm{Ad}(u)$ for a unitary $u\in M_n(\mathbb C)$ denotes the automorphism of $M_n(\mathbb C)$ such that $\mathrm{Ad}(u)M=uMu^*$.


\subsubsection{Endomorphism bundles $\mathscr{M}^{(n)}$} Let $\mathscr V$ be a line bundle over $X$ and $\alpha : X \to X$ a homeomorphism. With
\[
\mathscr{V}^{(n)} :=(\alpha^{n-1})^*\mathscr{V} \otimes (\alpha^{n-2})^*\mathscr{V} \otimes \cdots \otimes \alpha^*\mathcal{\mathscr{V}} \otimes  \mathscr{V}
\]
as in \eqref{eq:LBtensor}, let $\mathscr{D}^{(m,n)}$ and $\mathscr{D}^{(n)}$ be vector bundles over $X$ given by
\begin{align*}
\mathscr{D}^{(n)} &:=\mathscr{V}^{(0)} \oplus \mathscr{V}^{(1)} \oplus \mathscr{V}^{(2)} \oplus \cdots \oplus \mathscr{V}^{(n-1)} ,\\ \mathscr{D}^{(m,n)}&:= \mathscr{V}^{(m)} \oplus \mathscr{V}^{(m+1)} \oplus \mathscr{V}^{(m+2)} \oplus \cdots \oplus \mathscr{V}^{(m+n-1)}, 
\end{align*}
and let 
\begin{align*}
\mathscr{M}^{(n)}: = End(\mathscr{D}^{(n)}).
\end{align*}
Since  $\mathscr{D}^{(n)}$ is of rank $n$, the bundle $\mathscr{M}^{(n)}$ has fibre $\mathscr M^{(n)}_x=End(\mathscr D^{(n)}_x)$ isomorphic to $M_n(\mathbb{C})$.


\subsubsection{Local trivialization of $\mathscr M^{(n)}$} \label{sec:loctriv}

We fix a finite open cover $\mathcal U=\{U_j\}_j$ of $X$ such that $\mathscr V$ is trivial over $U_j$ with a chart map  
\[h_{U_j}:U_j\times \mathbb C\to \mathscr{V}|_{U_j}\] for each  $U_j\in \mathcal U$. 
From the atlas
$\{(U_j, h_{U_j})\mid U_j\in \mathcal U \}$ we obtain the transition functions $g_{U_i,U_j}:U_i\cap U_j\to U(1)$ such that 
\begin{equation}\label{h transition}
h_{U_i}(x,\lambda)=h_{U_j}(x,g_{U_j,U_i}(x)\lambda)   
\end{equation}
for all $x\in U_i\cap U_j$ and  $\lambda\in \mathbb C$. 
Fix $n\in \mathbb N$ and  let
\[\mathcal U^{(n)}:=\{(U_0,U_1,\dots,U_{n-1})\mid U_j\in \mathcal U \text{ and } \cap_{j=0}^{n-1}\alpha^{-j}(U_j)\neq \emptyset\}.\] 
Note that if $n=1$, then $\mathcal U^{(1)}=\mathcal U$, and that 
\[x\in \cap_{j=0}^{n-1}\alpha^{-j}(U_j)\ \Leftrightarrow \,\alpha^j(x)\in U_j\ \text{ for all } j=1, \dots, n-1.\] 
For convenience, we say that $x$ {\it is contained in} $\mathbf U=(U_0, \dots, U_{n-1})\in \mathcal U^{(n)}$ if  $x\in\cap_{j=0}^{n-1}\alpha^{-j}(U_j)$. In this case,  we simply write $x\in \mathbf U$:
\begin{align*}
 x\in \mathbf U=(U_0, \dots, U_{n-1}) &\ \Leftrightarrow \ x\in \cap_{j=0}^{n-1}\alpha^{-j}(U_j)\\ 
 &\ \Leftrightarrow\ \alpha^j(x)\in U_j\ \text{for all } j=0, \dots, n-1.  
\end{align*}
For  $\mathbf U\in\mathcal U^{(n)}$ and $1\leq l\leq n$, the map 
\[v^{(l)}_{\mathbf U}:\cap_{j=0}^{n-1}\alpha^{-j}(U_j)\to \mathscr V^{(l)}|_{\cap_{j=0}^{n-1}\alpha^{-j}(U_j)}\] given by 
\[v^{(l)}_{\mathbf U}(x):=h_{U_{l-1}}(\alpha^{l-1}(x),1)\otimes \cdots\otimes h_{U_1}(\alpha(x),1)\otimes h_{U_0}(x,1)\] 
is a non-zero continuous local section of the line bundle $\mathscr V^{(l)}$ over the open set $\cap_{j=0}^{n-1}\alpha^{-j}(U_j)$. 
The set of all open sets $\{\cap_{j=0}^{n-1}\alpha^{-j}(U_j)\mid \mathbf U=(U_0,\dots, U_{n-1})\in \mathcal U^{(n)}\}$ is a finite open cover of $X$  
and forms an atlas of $\mathscr V^{(l)}$ together with the following maps 
\[(x,\lambda)\mapsto \lambda v^{(l)}_{\mathbf U}(x):\cap_{j=0}^{n-1}\alpha^{-j}(U_j)\times \mathbb C\to \mathscr V^{(l)}|_{\cap_{j=0}^{n-1}\alpha^{-j}(U_j)}.\]

For later use, we set up some notation: If $x\in \mathbf U\in \mathcal U^{(n)}$, then  $\alpha^i(x)\in (U_i, \dots, U_{n-1})\in \mathcal U^{(n-i)}$ for $i=0,\dots, n-1$. 
 Reflecting this point we will use the following notation:
\[\alpha^i(\mathbf U):=(U_i, \dots, U_{n-1})\in \mathcal U^{(n-i)}.\] 
Also for $1\leq m\leq n-1$ and $0\leq i\leq n-1$ with $1\leq m+i\leq n$, we write 
\begin{equation}\label{v section}
v^{(m)}_{\alpha^i(\mathbf U)}(\alpha^i(x)):=h_{U_{m+i-1}}(\alpha^{m+i-1}(x),1)\otimes\cdots\otimes h_{U_i}(\alpha^i(x),1).
\end{equation}

Since $v_{\mathbf U}^{(l)}$ is a non-zero local section of $\mathscr V^{(l)}$, each unit vector $v_{\mathbf U}^{(l)}(x)$  serves  as a basis element of the one dimensional space $\mathscr V^{(l)}_x$ for all $x\in \cap_{j=0}^{n-1}\alpha^{-j}(U_j)$. 
In turn, one has a basis of the $n$-dimensional space  $\mathscr{D}^{(n)}_x=\mathscr{V}^{(0)}_x \oplus \mathscr{V}^{(1)}_x   \oplus \cdots \oplus \mathscr{V}^{(n-1)}_x$ at $x\in \mathbf U \in \mathcal U^{(n)}$ given by 
\begin{align*}
e_{\mathbf U,0}(x) &:=(v_{\mathbf U}^{(0)}(x),\ 0,\ 0,\,\ldots\,,\ 0),\\
e_{\mathbf U,1}(x) &:=(0,\ v_{\mathbf U}^{(1)}(x),\ 0,\,\dots\, ,\ 0),\\ 
   &\ \ \vdots \\
e_{\mathbf U,n-1}(x) &:=(0,\ 0,\,\dots\,,\ 0, \ v_{\mathbf U}^{(n-1)}(x)),
\end{align*} 
with the convention that $v_{\mathbf U}^{(0)}(x):=1$. 
Then the map 
\[(x,(\lambda_0, \dots, \lambda_{n-1}))\mapsto \lambda_0 \, e_{\mathbf U,0}(x)+\cdots+\lambda_{n-1} e_{\mathbf U,n-1}(x)\] from $\cap_{j=0}^{n-1}\alpha^{-j}(U_j)\times\mathbb C^n$ onto $\mathscr D^{(n)}|_{\cap_{j=0}^{n-1}\alpha^{-j}(U_j)}$ is a chart map of $\mathscr D^{(n)}$ (by continuity of $e_{\mathbf U,i}$'s over $\cap_{j=0}^{n-1}\alpha^{-j}(U_j)$) from which we obtain a chart map 
\[
H_{\mathbf U}^{(n)} : \cap_{j=0}^{n-1}\alpha^{-j}(U_j) \times M_n(\mathbb{C}) \to \mathscr{M}_{\mathbf U}^{(n)}:=\mathscr{M}^{(n)}|_{\cap_{j=0}^{n-1}\alpha^{-j}(U_j)}
\]   
of the endomorphism bundle $\mathscr{M}^{(n)}=End(\mathscr D^{(n)})$ over $\cap_{j=0}^{n-1}\alpha^{-j}(U_j)$. 
Hence, for a linear map $\varsigma(x)\in \mathscr{M}^{(n)}_x$, $x\in \cap_{j=0}^{n-1}\alpha^{-j}(U_j)$, there exists a unique matrix  $\widetilde{\varsigma}_{\mathbf U}(x)\in M_n(\mathbb C)$ such that $H_{\mathbf U}^{(n)}(x, \widetilde{\varsigma}_{\mathbf U}(x))  = \varsigma(x)$. 
Namely, if  $\widetilde{\varsigma}_{\mathbf U}(x)$ is written as 
\[
\widetilde{\varsigma}_{\mathbf U}(x) = \scriptstyle{\begin{bmatrix}
a_{00}(x) & a_{01}(x) &a_{02}(x) & \cdots & a_{0(n-1)}(x) \\
a_{10}(x) &a_{11}(x) & a_{12}(x) & \cdots & a_{1(n-1)}(x) \\
a_{20}(x) &a_{21}(x) & a_{22}(x) & \cdots & a_{2(n-1)}(x) \\
\vdots  &\vdots  & \vdots  & \ddots & \vdots  \\
a_{(n-1)0}(x) &a_{(n-1)1}(x) & a_{(n-1)2}(x)  & \cdots & a_{(n-1)(n-1)}(x) 
\end{bmatrix}},
\]
then for each $j=0,\dots, n-1$, 
\[ \varsigma(x) e_{\mathbf U,j}(x)\\
 =  a_{0j}(x)e_{\mathbf U,0}(x) +a_{1j}(x)e_{\mathbf U,1}(x)  + \cdots + a_{(n-1)j}(x)e_{\mathbf U,n-1}(x).   
\]
When $x$ is fixed and clear in context, by abuse of notation, we simply write $\widetilde{\varsigma}_{\mathbf U}(x) = (H_{\mathbf U}^{(n)})^{-1}(\varsigma(x))$. 

\begin{definition} \label{def:induced triv}
    We call the atlas $\{\big(\cap_{j=0}^{n-1}\alpha^{-j}(U_j),H_{\mathbf U}^{(n)}\big)\}_{\mathbf U\in \mathcal U^{(n)}}$ the {\it local trivialization of $\mathscr M^{(n)}$ induced  from  the atlas}  $\{(U_j,h_{U_j})\}_{U_j\in \mathcal U}$ of $\mathscr V$.
\end{definition}

Now we consider what happens to the matrix representation of a linear map in $\mathscr M_x^{(n)}$ when we make a coordinate change in the vector space $\mathscr D^{(n)}_x$.  For this, let $\mathbf U, \mathbf V \in \mathcal{U}^{(n)}$, $x\in \mathbf U\cap \mathbf V$ and $0\leq l\leq n$. Then 
we have from (\ref{h transition}) that 
\begin{align*} 
&\ v^{(l)}_{\mathbf U}(x)\\
=&\ h_{U_{l-1}}(\alpha^{l-1}(x),1)\otimes \cdots\otimes h_{U_1}(\alpha(x),1)\otimes h_{U_0}(x,1)\\ 
=&\ \Pi_{j=0}^{l-1} g_{V_j,U_j}(\alpha^j(x))\cdot h_{V_{l-1}}(\alpha^{l-1}(x),1)\otimes \cdots\otimes h_{V_1}(\alpha(x),1)\otimes h_{V_0}(x,1)\\
=&\ g^{(l)}_{\mathbf V, \mathbf U}(x)\cdot v^{(l)}_{\mathbf V}(x),
\end{align*} 
where 
$g^{(l)}_{\mathbf V, \mathbf U}(x):=\Pi_{j=0}^{l-1}\, g_{V_j,U_j}(\alpha^j(x)).$
With the following diagonal unitary matrix
\begin{equation}\label{def:u-UV}
    u_{\mathbf{VU}}(x):= \diag(1, g_{\mathbf V, \mathbf U}^{(1)}(x),g_{\mathbf V, \mathbf U}^{(2)}(x),\ldots,g_{\mathbf V, \mathbf U}^{(n-1)}(x)),
\end{equation}
    it is easily checked that 
\begin{equation}\label{varsigma-U-V}
 \widetilde{\varsigma}_{\mathbf V}(x) = u_{\mathbf{VU}}(x) \widetilde{\varsigma}_{\mathbf U}(x) u_{\mathbf{VU}}^{-1}(x).   
\end{equation} 
In other words, 
\[H_{\mathbf U}^{(n)}(x, \widetilde{\varsigma}_{\mathbf U}(x))=H_{\mathbf V}^{(n)}(x, \mathrm{Ad}(u_{\mathbf{VU}}(x))\widetilde{\varsigma}_{\mathbf{U}}(x)),\]
 and   $x\mapsto \mathrm{Ad}(u_{\mathbf{VU}}(x)):\cap_{j=0}^{n-1}\alpha^{-j}(U_j\cap V_j)\to \Aut(M_n(\mathbb C))$ are the transition functions for the atlas 
 $\{\big(\cap_{j=0}^{n-1}\alpha^{-j}(U_j),H_{\mathbf U}^{(n)}\big)\}_{\mathbf U\in \mathcal U^{(n)}}$ of $\mathscr M^{(n)}$.

If  $\widetilde{\varsigma}_{\mathbf U}(x)$ is written as above,  
then  with  $(x)$ omitted,
\begin{align*} 
&\ \widetilde{\varsigma}_{\mathbf V}
=\ u_{\mathbf{VU}}
\widetilde{\varsigma}_{\mathbf U}
{  u_{\mathbf{VU}}}^{-1}\\
=&\ \scriptstyle{
\begin{bmatrix}
a_{00} & a_{01}{g_{\mathbf V,\mathbf U}^{(1)}}^{-1} & \cdots & a_{0(n-1)}{g_{\mathbf V,\mathbf U}^{(n-1)}}^{-1} \\
g_{\mathbf{V,U}}^{(1)}a_{10} & a_{11} &  \cdots & g_{\mathbf{V,U}}^{(1)}a_{1(n-1)}{g_{\mathbf V,\mathbf U}^{(n-1)}}^{-1} \\
\vdots  & \vdots  & \ddots  & \vdots   \\
g_{\mathbf V,\mathbf U}^{(n-1)}a_{(n-1)0} &  {g_{\mathbf V,\mathbf U}^{(n-1)}}a_{(n-1)2}{g_{\mathbf V, \mathbf U}^{(2)}}^{-1} &  \cdots & a_{(n-1)(n-1)}
\end{bmatrix}}.
\end{align*}
Note that the matrices $\widetilde{\varsigma}_{\mathbf U}(x)$ and $\widetilde{\varsigma}_{\mathbf V}(x)$ have the same diagonal elements, that is, the diagonal is invariant under the coordinate change from $\{e_{\mathbf U,i}(x)\}$ to $\{e_{\mathbf V,i}(x)\}$ for $x\in \mathbf U\cap \mathbf V$. Also if one of the matrices is $m$\textsuperscript{th} subdiagonal, then so is the other for $1\leq m\leq n-1$. 

By a {\it matrix representation} of a linear map $\varsigma(x)$ in  $\mathscr{M}^{(n)}_x$, we always mean a matrix $\widetilde{\varsigma}_{U}(x)=(H_{\mathbf U}^{(n)})^{-1}(\varsigma(x))$ for some $\mathbf U=(U_0,\dots,U_{n-1})\in \mathcal U^{(n)}$ with $x\in \mathbf U$ if not mentioned otherwise. We have the following remark from the above observations.

\begin{remark}\label{remark:matrix form}   Let $\varsigma\in \Gamma(\mathscr M^{(n)})$ and let $\widetilde{\varsigma}_{\mathbf U}(x)=(H_{\mathbf U}^{(n)})^{-1}(\varsigma(x))\in M_n(\mathbb C)$ be a matrix representation of $\varsigma(x)$ for $x$ contained in $\mathbf U\in \mathcal U^{(n)}$.  Then we have the following: 
\begin{enumerate}[left=0pt]
\item  The main diagonal is invariant: $\diag(\widetilde{\varsigma}_{\mathbf U}(x))=\diag(\widetilde{\varsigma}_{\mathbf V}(x))$ whenever $x\in \mathbf U\cap \mathbf V$ for $\mathbf{U,V}\in \mathcal U^{(n)}$. 
\item  If $\widetilde{\varsigma}_{\mathbf U}(x)$ is $m$\textsuperscript{th} subdiagonal, then so is  $\widetilde{\varsigma}_{\mathbf V}(x)$ whenever $x\in \mathbf U\cap \mathbf V$. 
\item  If $\varsigma(x)$ is represented as an  $m$\textsuperscript{th} (lower) subdiagonal matrix $\widetilde{\varsigma}_{\mathbf U}(x)=(H_{\mathbf U}^{(n)})^{-1}(\varsigma(x))$ on $\cap_{j=0}^{n-1}\alpha^{-j}(U_j)$ for $1\leq m\leq n-1$, then $\widetilde{\varsigma}_{\mathbf U}(x)$ is of the form 
 \vspace{-10pt}
 \[{}\qquad \ 
\widetilde{\varsigma}_{\mathbf U}(x) =\footnotesize{
\bbordermatrix{& & \cr 
\hfill 0 \hfill& 0 & 0  & \cdots & 0 & 0 & \cdots & 0  \cr
\hfill \vdots \hfill &\vdots & \vdots &   & \vdots &   \vdots & &\vdots   \cr
\hfill m-1 \hfill & 0 & 0  & \cdots & 0 & 0 & \cdots & 0 \cr
\hfill m \hfill & s_{m,0}(x)  & 0 & \cdots & 0 & 0 & \cdots & 0   \cr
\hfill m+1 \hfill & 0 &s_{m+1,1}(x)  & \cdots & 0& 0 & \cdots & 0   \cr
\hfill \vdots \hfill  &\vdots &\vdots    & \ddots & \vdots & \vdots & \cdots & \vdots   \cr
\hfill n-1 \hfill &0 & 0  & \cdots & s_{n-1,n-1-m}(x)& 0 & \cdots & 0   \cr
}},
\]
where each $s_{i,j}$ is a continuous function on $\cap_{j=0}^{n-1}\alpha^{-j}(U_j)$. 
In this case $\varsigma(x)$ maps each element $e_{\mathbf U,i}(x)$ in basis of $\mathscr{D}^{(n)}_x$ to
\begin{equation}\label{m-diagonal component}
	\varsigma(x) e_{\mathbf U,i}(x) =
\begin{cases}
	s_{m+i,i}(x)e_{\mathbf U,m+i}(x), & \text{for $0\leq i<n-m$},\\
    0, &  \text{for $n-m\leq i\leq n-1$}.
\end{cases}	
\end{equation}
We recall the definition of the basis element $e_{\mathbf U,i}(x)$ and see that $\varsigma(x)$ maps each $ \mathscr{V}^{(i)}_x$ to  $ \mathscr{V}^{(i+m)}_x $ for $0\leq i<n-m$, and furthermore, the linear map 
$$\varsigma(x)|_{\mathscr{V}^{(i)}_x}: \mathscr{V}^{(i)}_x \to \mathscr{V}^{(m+i)}_x =\mathscr{V}_{\alpha^{i}(x)}^{(m)} \otimes \mathscr{V}^{(i)}_x$$ 
determines a vector $\varsigma^{(i)}(x) \in \mathscr{V}_{\alpha^{i}(x)}^{(m)}$ such that 
\begin{equation}\label{s_k^i}
\varsigma(x)|_{\mathscr{V}^{(i)}_x}: v \mapsto \varsigma^{(i)}(x)\otimes v,\quad v \in  \mathscr{V}^{(i)}_x,
\end{equation}  
since $ \mathscr{V}^{(i)}_x$ and $\mathscr{V}_{\alpha^{i}(x)}^{(m)}$ are all one dimensional. Using the equation (\ref{m-diagonal component}) we see that
\begin{equation}
 \varsigma^{(i)}(x) =   s_{m+i,i}(x) v^{(m)}_{\alpha^i(\mathbf U)}(\alpha^i(x))
\end{equation}
since $v^{(m)}_{\alpha^i(\mathbf U)}(\alpha^i(x))=h_{U_{m+i-1}}(\alpha^{m+i-1}(x),1)\otimes\cdots\otimes h_{U_i}(\alpha^i(x),1)$  as defined in (\ref{v section}).
\end{enumerate}   
\end{remark}

Note that 
$\Gamma(\mathscr M^{(n)})$ is an algebra with respect to multiplication defined by pointwise composition of endomorphisms at each fibre. For $\varsigma\in \Gamma(\mathscr M^{(n)})$ and $x\in \mathbf U\cap \mathbf V$, we have 
\begin{align*}
 H_{\mathbf U}^{(n)}(x, (H_{\mathbf U}^{(n)})^{-1}(\varsigma(x))^*) &=\ H_{\mathbf U}^{(n)}(x, \widetilde{\varsigma}_{\mathbf U}(x)^*)\\
 &=\ H_{\mathbf V}^{(n)}(x, u_{\mathbf{VU}}\widetilde{\varsigma}_{\mathbf U}(x)^*u_{\mathbf{VU}}^*) \\ 
 &=\ H_{\mathbf V}^{(n)}(x, (u_{\mathbf{VU}}\widetilde{\varsigma}_{\mathbf U}(x)u_{\mathbf{VU}})^*) \\ 
 &=\ H_{\mathbf V}^{(n)}(x, \widetilde{\varsigma}_{\mathbf V}(x)^*) \\ 
 &=\ H_{\mathbf V}^{(n)}(x, (H_{\mathbf V}^{(n)})^{-1}(\varsigma(x))^*), 
\end{align*}
hence $\varsigma^*$ is well defined by $\varsigma^*(x):=H_{\mathbf U}^{(n)}(x, (H_{\mathbf U}^{(n)})^{-1}(\varsigma(x))^*)$. The $^*$-algebra $\Gamma(\mathscr M^{(n)})$ is in fact a $C^*$-algebra with respect to the norm 
\[
\|\varsigma\|:=\sup_{x\in X} \|(H_{\mathbf U}^{(n)})^{-1}(\varsigma(x))\|_{M_n(\mathbb C)}.
\]
Note that $\|(H_{\mathbf U}^{(n)})^{-1}(\varsigma(x))\|_{M_n(\mathbb C)}=\|(H_{\mathbf V}^{(n)})^{-1}(\varsigma(x))\|_{M_n(\mathbb C)}$, so the norm above does not depend on the choice of the set $\mathbf U$.

\section{Embedding \texorpdfstring{$\mathcal{O}(\E_Y)$}{} into a direct sum of homogeneous algebras}\label{sec:embedding}

From now on we fix an infinite compact metric space $X$, a minimal homeomorphism $\alpha : X \to X$, a line bundle $\mathscr{V}$  over $X$, and a closed subset $Y \subset X$ with non-empty interior. Put $\E := \Gamma(\mathscr{V}, \alpha)$ and $\E_Y := C_0(X \setminus Y) \E$.

As in Section~\ref{sec:rsh}, we denote by $0 < r_1 < \dots < r_K$ the finitely many first return times to $Y$ and set
\[ Y_k := \{ y \in Y \mid r_Y(y) = r_k\}.\]

\begin{definition}\label{def:M-k}
Fix a finite open cover $\mathcal U$  of $X$ as in Section~\ref{sec:loctriv}. For $k=1,2,\ldots,K$, we define $\mathscr{M}_k$ to be the restriction of the endomorphism bundle $\mathscr{M}^{(r_k)}$ to $\overline{Y_k}$, namely
\[ \mathscr{M}_k:=\mathscr{M}^{(r_k)}|_{\overline{Y_k}}=End(\mathscr{D}^{(r_k)})|_{\overline{Y_k}}.
\]    
For convenience, we write $H_{\mathbf U}^{(k)}:=H_{\mathbf U}^{(r_k)}$, $\mathbf U\in \mathcal U^{(r_K)}$, for the rest of the paper. 
\end{definition}

In  this section we show that there exists an embedding $\pi$ of the orbit-breaking algebra $\mathcal O(\E_Y)$ into  $\Gamma(\mathscr M_1)\oplus\cdots\oplus\Gamma(\mathscr M_K)$, where $\Gamma(\mathscr M_k)$ is the $\mathrm C^*$-algebra of continuous sections of the  endomorphism bundle $\mathscr M_k$. We begin by constructing a $\mathrm{C}^*$-homomorphism $\pi_k$ from $\mathcal{O}(\mathcal{E}_Y)$ into $\Gamma(\mathscr M_k)$ for every $k=1, \dots, K$.  We end the section by proving that $\pi_1$ is surjective while it is not true for $\pi_k$ when $2 \leq k \leq K$ unless $Y_k$ is closed.

To construct a $*$-homomorphism $\pi_k$ from the Cuntz--Pimsner algebra $\mathcal{O}(\E_Y)$ into the $\mathrm C^*$-algebra $\Gamma(\mathscr M_k)$, we define a covariant representation $(\pi_k, \tau_k)$ of $\E_Y$ into the $\mathrm C^*$-algebra $\Gamma(\mathscr M_k)$ which gives rise to the integrated  $^*$-homomorphism of $\mathcal O(\E_Y)$ into $\Gamma(\mathscr M_k)$ which  we will also denote by $\pi_k$.

\begin{definition} \label{def:covrep} For $f\in C(X)$, $\xi\in \E_Y$, and $x\in \overline{Y_k}$, let $\pi_k(f)(x)$ and $\tau_k(\xi)(x)$ be the linear maps on $\mathscr D_x^{(r_k)}$ given  by 
\begin{align*}
\pi_k(f)(x)(a_0, \ldots, a_{r_k-1}) &=   (f(x) a_0,  f(\alpha(x))a_1, \ldots, f(\alpha^{r_k-1}(x)) a_{r_k-1}) \\ \tau_k(\xi)(x)(a_0, \ldots, a_{r_k-1})  \\ =& ( 0, \xi(x)\otimes a_0, \xi(\alpha(x))\otimes a_1,\ldots, \xi(\alpha^{r_k-2}(x)) \otimes a_{r_k-2})   
\end{align*} 
for  $(a_0, a_1, \ldots, a_{r_k-1}) \in \mathscr D^{(r_k)}_x=\mathscr{V}^{(0)}_x \oplus \mathscr{V}^{(1)}_x \oplus \cdots \oplus \mathscr{V}^{(r_k-1)}_x$, respectively. 
Here we mean $\xi(x)\otimes a_0:=\xi(x) a_0$. 
By continuity of $f$ and $\xi$, this extends to two maps $\pi_k:C(X)\to \Gamma(\mathscr M_k)$ and $\tau_k:\E_Y\to \Gamma(\mathscr M_k)$.
\end{definition} 

Let $x\in Y_k\cap \mathbf U$, $\mathbf U=(U_0, \dots, U_{r_K-1})\in \U^{(r_K)}$, and consider  the basis $\{e_{\mathbf U,0}(x),\dots, e_{\mathbf U,r_k-1}(x)\}$ of $\mathscr D^{(r_k)}_x$.  If $\xi\in \E_Y$, then for each $j=0,\dots,r_k-2$, the scalar $\widetilde\xi(\alpha^j(x))\in \mathbb C$ with 
\begin{equation}\label{matrix entries}
\xi(\alpha^j(x))=\widetilde\xi(\alpha^j(x))h_{U_j}(\alpha^j(x), 1),
\end{equation}
satisfies
\begin{equation}\label{tau_k shift}
\tau_k(\xi)(x)e_{\mathbf U,j}(x)=\widetilde\xi(\alpha^j(x))e_{\mathbf U,j+1}(x),
\end{equation}
and for the unitary $U_x:\mathbb C^{r_k}\to \mathscr D^{(r_k)}_x$, 
\[U_x(\lambda_0, \dots, \lambda_{r_k-1})= \lambda_0 e_{\mathbf U,0}(x)+\cdots +\lambda_{r_k-1} e_{\mathbf U,r_k-1}(x), \] 
we have the  matrix representation 
\begin{equation}\label{tau_k matrix}
M_{\tau_k(\xi)}(x)=
\begin{bmatrix}
0 & 0 & \cdots & 0 & 0 & 0 \\
\widetilde{\xi}(x)  & 0 & \cdots & 0 & 0 & 0 \\
0 & \widetilde{\xi}(\alpha(x))  & \ddots & 0 & 0 & 0 \\
\vdots & \vdots & \ddots & \vdots & \vdots & \vdots \\
0 & 0 & \cdots & \widetilde{\xi}(\alpha^{r_k-3}(x))  & 0 & 0 \\
0 & 0 & \cdots  & 0 & \widetilde{\xi}(\alpha^{r_k-2}(x))  & 0
\end{bmatrix}
\end{equation} 
of $U_x^*\tau_k(\xi)(x)U_x$, so 
\[\tau_k(\xi)(x)=U_x M_{\tau_k(\xi)}(x) U_x^*.\]
This matrix representation will be used in the subsequent discussions along with the fact that the matrix depends on the choice of $\mathbf U\in \mathcal U^{(r_K)}$ with $x\in \mathbf U$, but the matrix entries differ only by multiplication by scalars of modulus one. 
For $f\in C(X)$, the corresponding matrix representation is clearly
\begin{equation}\label{pi_k matrix}
M_{\pi_k(f)}(x)=\diag(f(x), f(\alpha(x)), \dots, f(\alpha^{r_k-1}(x))).
\end{equation}

It is rather obvious that $\pi_k$ is in fact a $^*$-homomorphism, and we may write $\pi_k(f)(x)$ as 
\[
\pi_k(f)(x)= \diag(f(x), f(\alpha(x)), \ldots, f(\alpha^{r_k-1}(x)) ).    
\]
One can show that $\tau_k(\xi)^*\in \Gamma(\mathscr M_k)$ is given by at $x\in \overline{Y_k}$, 
\begin{equation}\label{xi^*}
\tau_k(\xi)^*(x)(a_0, \ldots, a_{r_k-1}) = ( {\xi(x)}^* \otimes a_1, \ldots, \xi(\alpha^{r_k-2}(x))^* \otimes a_{r_k-1}, 0),     
\end{equation}
where  
\[
\xi(x)^*\otimes a_1:=\langle \xi(x),a_1\rangle_{\mathscr V_x}
\]
and 
\[\xi(\alpha^{i-1}(x))^*\otimes a_i:=\langle \xi(\alpha^{i-1}(x)),a^{i-1}_i\rangle_{\mathscr V_{\alpha^{i-1}(x)}}\, a_i^{i-2}\otimes\cdots\otimes a_i^1\otimes a_i^0\in \mathscr V_x^{(i-1)}\]
for $a_i=a_i^{i-1}\otimes\cdots \otimes a_i^1\otimes a_i^0\in \mathscr V^{(i)}_x$, $2\leq i\leq r_k-1$.

\begin{lemma}\label{lemma:pi_and_tau} Let $\mathcal{E}_Y := C_0(X \setminus Y) \mathcal{E}$. Then we have the following.
\begin{enumerate}[label=(\alph*)]
\item[$(a)$] $f \mapsto \pi_k(f) : C(X) \to \Gamma(\mathscr{M}_k)$ is a $^*$-homomorphism.
\item[$(b)$] $\xi \mapsto \tau_k(\xi) : \mathcal{E}_Y \to \Gamma(\mathscr{M}_k)$ is a linear map.
\item[$(c)$] $\pi_k(\langle \xi, \eta \rangle_{\mathcal{E}} ) = \tau_k(\xi)^* \tau_k(\eta)$, for  $\xi,\eta\in \E_Y$. 
\item[$(d)$] $\tau_k(f\xi) = \pi_k(f) \tau_k(\xi) = \tau_k(\xi) \pi_k(f\circ \alpha)$.
\item[$(e)$] $\pi_k(_{\mathcal{E}_Y}\langle \xi, \eta \rangle ) = \tau_k(\xi)\tau_k(\eta)^*$, where $_{\mathcal{E}_Y}\langle \xi, \eta \rangle:=\langle \eta,\xi\rangle_{\mathcal E}\circ\alpha^{-1}$ is the left inner product   inherited from $\mathcal{E}_Y \subset \mathcal{E}$. 
\end{enumerate}	
\end{lemma}

\begin{proof} (a) This is clear.  

\noindent (b) This is also clear.

\noindent 
(c)	We have
\begin{align*}
\pi_k(\langle \xi, \eta \rangle_{\mathcal{E}} )(x) & = diag( \langle \xi, \eta \rangle_{\mathcal{E}} (x),  \langle \xi, \eta \rangle_{\mathcal{E}} (\alpha(x)), \ldots, \langle \xi, \eta \rangle_{\mathcal{E}} (\alpha^{r_k-1}(x)) ), 
\end{align*}
and applying (\ref{xi^*}), 
\begin{align*}
&\ \phantom{ = } \tau_k(\xi)^* \tau_k(\eta)(x) (a_0, a_1, \ldots, a_{r_k-1})\\
&=  \tau_k(\xi)^*(x) ( 0,\eta(x)\otimes a_0, \eta(\alpha(x)) \otimes a_1, \ldots, \eta(\alpha^{r_k-2}(x)) \otimes a_{r_k-2})\\	
&=   ( \xi(x)^*\otimes\eta(x)\otimes a_0, \xi(\alpha(x))^*\otimes\eta(\alpha(x)) \otimes a_1, \\
&\qquad\qquad\qquad \ldots, \xi(\alpha^{r_k-2}(x))\otimes\eta(\alpha^{r_k-2}(x)) \otimes a_{r_k-2},0)\\	
&=   ( \langle \xi, \eta \rangle_{\mathcal{E}} (x) a_0, \langle \xi, \eta \rangle_{\mathcal{E}} (\alpha(x)) a_1, \ldots, \langle \xi, \eta \rangle_{\mathcal{E}} (\alpha^{r_k-2}(x))  a_{r_k-2}, 0)\\	
&=   ( \langle \xi, \eta \rangle_{\mathcal{E}} (x) a_0, \ldots, \langle \xi, \eta \rangle_{\mathcal{E}} (\alpha^{r_k-2}(x))  a_{r_k-2}, \langle \xi, \eta \rangle_{\mathcal{E}} (\alpha^{r_k-1}(x))  a_{r_k-1})\\
&= \diag( \langle \xi, \eta \rangle_{\mathcal{E}} (x),  \langle \xi, \eta \rangle_{\mathcal{E}} (\alpha(x)), \ldots, \langle \xi, \eta \rangle_{\mathcal{E}} (\alpha^{r_k-1}(x)) ) (a_0, a_1, \ldots, a_{r_k-1}),
\end{align*}
since $\langle \xi, \eta \rangle_{\mathcal{E}} (\alpha^{r_k-1}(x))  a_{r_k-1}=0$.

\noindent 
(d) If $f\in C(X)$ and $\xi\in \E_Y$, then for any  $(a_0, \dots, a_{r_k-1})\in \mathscr D^{(r_k)}_x$, we have 
\begin{align*}
&\ (\pi_k(f)\tau_k(\xi))(x)(a_0, a_1, \dots, a_{r_k-1})\\
=&\ \pi_k(f)(x)(0, \xi(x)\otimes a_0, \xi(\alpha(x))\otimes a_1, \dots, \xi(\alpha^{r_k-2}(x))\otimes a_{r_k-2})\\
=&\ (0, f(\alpha(x))\xi(x)\otimes a_0, f(\alpha^2(x))\xi(\alpha(x))\otimes a_1, \\ & \dots, f(\alpha^{r_k-1}(x))\xi(\alpha^{r_k-2}(x))\otimes a_{r_k-2})\\
=&\ (0, (\xi f\circ\alpha)(x)\otimes a_0, \dots, (\xi f\circ\alpha)(\alpha^{r_k-2}(x))\otimes a_{r_k-2})\\
=& \ (0, (f\xi)(x)\otimes a_0, (f\xi)(\alpha(x))\otimes a_1, \dots, (f\xi)(\alpha^{r_k-2}(x))\otimes a_{r_k-2})\\
=&\ \tau_k(f\xi)(a_0, a_1, \dots, a_{r_k-1}).
\end{align*}

\noindent
(e) Since ${}_{\mathcal{E}_Y}\langle \xi, \eta \rangle  (x)=\langle \eta,\xi\rangle_{\mathcal E}\circ\alpha^{-1}(x)=0$ for $x\in \overline{Y_k}\subset Y$ and $\xi,\eta\in \mathcal E_Y$, 
\begin{align*}
&\ \pi_k({}_{\mathcal{E}_Y}\langle \xi, \eta \rangle )(x)(a_0,\dots, a_{r_k-1})\\ 
=&\ ( {}_{\mathcal{E}_Y}\langle \xi, \eta \rangle  (x)a_0,  {}_{\mathcal{E}_Y}\langle \xi, \eta \rangle  (\alpha(x))a_1, \ldots, {}_{\mathcal{E}_Y}\langle \xi, \eta \rangle  (\alpha^{r_k-1}(x))a_{r_k-1} )	\\
=&\ ( 0,  {}_{\mathcal{E}_Y}\langle \xi, \eta \rangle  (\alpha(x))a_1, \ldots, {}_{\mathcal{E}_Y}\langle \xi, \eta \rangle  (\alpha^{r_k-1}(x)) a_{r_k-1})\\
=&\ (0, \langle \eta, \xi \rangle_\E(x)a_1, \dots, \langle \eta, \xi \rangle_\E  (\alpha^{i-1}(x))a_i,\dots, \langle \eta, \xi \rangle_\E  (\alpha^{r_k-2}(x))a_{r_k-1}).
\end{align*} 
On the other hand, with $a_i=a_i'\otimes a_i^{(i-1)}\in \mathscr V_{\alpha^{i-1}(x)}\otimes \mathscr V^{(i-1)}_x$ for $1\leq i\leq r_k-1$,
\begin{align*}
&\ \tau_k(\xi)(x)\tau_k(\eta)^*(x) (a_0, a_1, \dots, a_{r_k-1}) \\
=&\  \tau_k(\xi)(x) ( \langle\eta(x), a_1\rangle_{\mathscr V_x}, \ldots, \langle\eta(\alpha^{r_k-2}(x)),a_{r_k-1}'\rangle_{\mathscr V_{\alpha^{r_k-2}(x)}}a_{r_k-1}^{(r_k-2)}, 0)\\
=&\ (0, \xi(x)\otimes\langle\eta(x), a_1\rangle_{\mathscr V_x}, \dots, \xi(\alpha^{r_k-2}(x))\otimes\langle\eta(\alpha^{r_k-2}(x)),a_{r_k-1}'\rangle_{\mathscr V_{\alpha^{r_k-2}(x)}}a_{r_k-1}^{(r_k-2)}). 
\end{align*}
 Thus it is enough to see  that for each $1\leq i\leq r_k-1$,
  \[\langle \eta, \xi \rangle_\E  (\alpha^{i-1}(x))a_i=\xi(\alpha^{i-1}(x))\otimes \langle \eta(\alpha^{i-1}(x)),a_i'\rangle_{\mathscr V_{\alpha^{i-1}(x)}}a_i^{(i-1)}.\]  
  But the left hand side
\[\langle \eta, \xi \rangle_\E  (\alpha^{i-1}(x))a_i
= \langle \eta(\alpha^{i-1}(x)), \xi(\alpha^{i-1}(x)) \rangle_{\mathscr V_{\alpha^{i-1}(x)}} a_i'\otimes a_i^{(i-1)}
\] must equal to the right hand side
\begin{align*}
    &\xi(\alpha^{i-1}(x))\otimes \langle \eta(\alpha^{i-1}(x)),a_i'\rangle_{\mathscr V_{\alpha^{i-1}(x)}}a_i^{(i-1)}
\\ &=\langle \eta(\alpha^{i-1}(x)),a_i'\rangle_{\mathscr V_{\alpha^{i-1}(x)}}\xi(\alpha^{i-1}(x))\otimes a_i^{(i-1)} 
\end{align*}
simply  because of the fact that $\langle v_1,v_2\rangle_{\mathscr V_y} v_3=\langle v_1,v_3\rangle_{\mathscr V_y} v_2$ for any $y\in X$ and $v_1, v_2, v_3\in \mathscr V_y$ which can be seen by considering their inverse images under a chart map $h_{\mathbf U}$ for $\mathbf U\in \mathcal U^{(1)}$ with $y\in \mathbf U$. 
\end{proof}

\begin{lemma}\label{lemma:covariant rep}
\begin{enumerate}[label=(\alph*)]
\item The pair $(\pi_k, \tau_k)$ is a covariant representation of the $C(X)$-correspondence $\mathcal{E}_Y$, hence there exists a $^*$-homomorphism 
\[	\mathcal{O}(\mathcal{E}_Y) \to \Gamma(\mathscr{M}_k)
\]
which we also denote by $\pi_k$. 
\item $\Gamma(\mathscr{M}_k)$ admits a gauge action $\beta:  \mathbb{T} \times \Gamma(\mathscr{M}_k) \to \Gamma(\mathscr{M}_k), (z, \xi) \mapsto \beta_z(\xi)$.  
\end{enumerate}
\end{lemma}

\begin{proof}
(a) Lemma~\ref{lemma:pi_and_tau} (a)-(d) imply that $(\pi_k, \tau_k)$ is a representation of $\E_Y$ on $\Gamma(\mathscr M_k)$. It remains to show that the representation is covariant. Let $\varphi_{\mathcal{E}_Y}$ be the left action of $C(X)$ on $\mathcal{E}_Y$ given by restriction of the structure map $\varphi_\E : C(X)\to \mathcal{L}(\mathcal{E})$.  Let 
\[
J_{\mathcal{E}_Y} := \varphi_{\mathcal{E}_Y}^{-1} (\mathcal{K}(\mathcal{E}_Y)) \cap (\ker \varphi_{\mathcal{E}_Y})^\perp.
\]
We have \[
\theta_{\xi,\eta} (\zeta) = \xi \langle \eta, \zeta \rangle_{\mathcal{E}}  = {}_{\mathcal{E}}\langle \xi, \eta \rangle \zeta = \varphi_{\mathcal{E}_Y} (_{\mathcal{E}}\langle \xi, \eta \rangle )\zeta.
\]
for any $\xi, \eta, \zeta \in \E_Y$, from which it follows that
\[
\theta_{\xi,\eta} = \varphi_{\mathcal{E}_Y} ( _{\mathcal{E}}\langle \xi, \eta \rangle),\qquad \xi,\eta\in \mathcal{E}_Y.
\]
Hence $(\pi_k,\tau_k)$ is a covariant representation since Lemma \ref{lemma:pi_and_tau}(e) gives
\[	\pi_k (f) = \psi_{\tau_k}(\varphi_{\mathcal{E}_Y}(f)),\qquad f \in C_0(X \setminus Y),
\]
where $ \psi_{\tau_k} : \mathcal{K}(\mathcal{E}_Y) \to \Gamma(\mathscr{M}_k)$ is the $^*$-homomorphism given by $ \psi_{\tau_k}(\theta_{\xi,\eta})= \tau_k(\xi) \tau_k(\eta)^*$. It follows that $J_{\E_Y} = C_0 (X\setminus Y)$.

Now the existence of the $^*$-homomorphism $\pi_k : \mathcal{O}(\mathcal{E}_Y) \to \Gamma(\mathscr{M}_k)$  follows from the universal property of $\mathcal{O}(\E_Y)$.

(b) For $z \in \mathbb{T}$, let $\beta_z: \Gamma(\mathscr{M}_k) \to \Gamma(\mathscr{M}_k)$ be the $*$-homomorphism  given by $\beta_z=Ad(U_z)$ on each fibre $\mathscr M^{k}_x$ at $x$ for a unitary map $U_z$ which is an endomorphism on each fibre $\mathscr D^{(r_k)}_x$ at $x$ by
\[
U_z : (a_0, a_1, a_2, \ldots, a_{r_k-1})  \mapsto ( a_0,  za_1, z^2 a_2,\ldots, z^{r_k-1} a_{r_k-1}).
\]
We see that for $f\in C(X)$, 
\begin{align*}	
& \ U_z\pi_k(f)(x) U_z^* (a_0, a_1, \, \ldots,\, a_{r_k-1}) \\  
= &\  U_z\pi_k(f)(x)( a_0,  \overline{z}a_1, \overline{z}^2 a_2,\, \ldots,\, \overline{z}^{r_k-1} a_{r_k-1})\\
=&\  U_z(f(x) a_0,  f(\alpha(x))\overline{z}a_1, f(\alpha^2(x))\overline{z}^2 a_2, \, \ldots,\, f(\alpha^{r_k-1}(x)) \overline{z}^{r_k-1}a_{r_k-1})\\
=& \ (f(x) a_0,  zf(\alpha(x))\overline{z}a_1, z^2f(\alpha^2(x))\overline{z}^2 a_2, \, \ldots,\, z^{r_k-1}f(\alpha^{r_k-1}(x)) \overline{z}^{r_k-1}a_{r_k-1})\\
=&\ (f(x) a_0,  f(\alpha(x))a_1, f(\alpha^2(x)) a_2, \, \ldots,\, f(\alpha^{r_k-1}(x)) a_{r_k-1})\\
=& \ \pi_k(f)(x)
\end{align*}
and for $\xi \in \mathcal{E}_Y$,
\begin{align*}	
&\ U_z \tau_k(\xi)(x) U_z^* (a_0, a_1, \, \ldots,\, a_{r_k-1})   \\ 
= &\  U_z \tau_k(\xi)(x)( a_0,  \overline{z}a_1, \overline{z}^2 a_2,\, \ldots,\, \overline{z}^{r_k-1} a_{r_k-1})\\
=& \ U_z( 0, \xi(x) \otimes a_0, \xi(\alpha(x)) \otimes \overline{z}a_1, \, \ldots,\, \xi(\alpha^{r_k-2}(x)) \otimes \overline{z}^{r_k-2}a_{r_k-2})\\
=& \ ( 0, z\xi(x) \otimes a_0, z^2\xi(\alpha(x)) \otimes \overline{z}a_1,\, \ldots,\, z^{r_k-1}\xi(\alpha^{r_k-2}(x)) \otimes \overline{z}^{r_k-2}a_{r_k-2})\\
=&\ ( 0, z\xi(x) \otimes a_0, \, \ldots,\, z\xi(\alpha^{r_k-2}(x)) \otimes a_{r_k-2})\\
=& \ z \tau_k(\xi)(x).
\end{align*}
Thus the representation $(\pi_k, \tau_k)$ admits a gauge action.
\end{proof}

\begin{proposition}\label{prop:gauge action} Let $\pi_k:\mathcal O(\E_Y)\to \Gamma(\mathscr M_k)$ be the $^*$-homomorphism obtained in Lemma~\ref{lemma:covariant rep} for $1\leq k\leq K$. Then 
\[\pi:=\oplus_{k=1}^K\pi_k: \mathcal O(\E_Y)\to \Gamma(\mathscr M_1)\oplus\cdots\oplus \Gamma(\mathscr M_K)\]
is an injective $^*$-homomorphism.
\end{proposition}

\begin{proof} 
The gauge actions $z\mapsto \beta_{k,z}$ on $\pi_k(\mathcal O(\E_Y))$, $k=1, \dots, K$, obtained in Lemma~\ref{lemma:covariant rep}, define a gauge action on $\pi(\mathcal O(\E_Y))$ 
\[z\mapsto \beta_z:=\beta_{1,z}\oplus\cdots\oplus \beta_{K,z}.\] 
Thus to show that $\pi$ is injective, it is enough to show injectivity on $C(X)$ \cite[Theorem 6.4]{Katsura2004}. 

Let $f\in C(X)$ be a function such that 
$\pi(f)=\pi_1(f)\oplus\cdots\oplus \pi_K(f)=0$. 
Then 
\[\pi_k(f)(x)= \diag(f(x), f(\alpha(x)), \dots, f(\alpha^{r_k-1}(x))=(0, \dots, 0)\] for every $x\in \overline{Y_k}$ and $1\leq k\leq K$. 
Sine $X=\cup_{k=1}^K \cup_{j=1}^{r_k-1} \alpha^j(\overline{Y_k})$ by (\ref{X:union}), we have $f=0$ as desired.
\end{proof}

\begin{remark}\label{remark:identification}
If $a=\xi_1\xi_2\cdots \xi_k \xi_{k+1}^*\cdots \xi_{k+l}^*\in \mathcal O(\E_Y)$ for $\xi_i\in \E_Y$, and $k+l$ is the minimal length for any possible such expressions of $a$, then we say that $\deg(a)=k-l$. Then 
the span of the elements of degree $n$ for $n\in \mathbb Z$ is dense in $\mathcal O(\E_Y)$, namely 
\[
\mathcal{O}(\mathcal{E}_Y)= \overline{\oplus_{n\in \mathbb{Z}} \overline{E_n}},
\]
where $E_n$ denotes the  $C(X)$-linear span of the elements of degree $n$. Note that 
$E_0 = C(X)$, and for  $n\geq 1$ 
\[ \xi_1\otimes\cdots\otimes \xi_n\mapsto \xi_1\cdots\xi_n\] extends to an isometric  linear map of  $\E_Y^{\otimes n}$ onto  $E_n$ (see Remark 4.1 of \cite{AAFGJSV2024}). Thus  we may identify $\E_Y^{\otimes n}$ with the subspace $E_n$ of $\mathcal O(\E_Y)$. 

Let $\xi=\xi_1\xi_2\in E_2$ for $\xi_1, \xi_2\in \E_Y$. Then for $(a_0, a_1, \ldots, a_{r_k-1}) \in \mathscr{V}^{(0)}_x \oplus \mathscr{V}^{(1)}_x \oplus \mathscr{V}^{(2)}_x \oplus \cdots \oplus \mathscr{V}^{(r_k-1)}_x$ and $x\in \overline{Y_k}$, since $\pi_k$ is a homomorphism, we have 
\begin{align*}
\lefteqn{\pi_k(\xi_1\xi_2)(x)(a_0, a_1, \ldots, a_{r_k-1})} \\
&=   \pi_k(\xi_1)(x)\pi_k(\xi_2)(x)(a_0, a_1, \ldots, a_{r_k-1})  \\
&= \pi_k(\xi_1)(x)(0,\xi_2(x)\otimes a_0, \dots, \xi_2(\alpha^{r_k-2}(x))\otimes a_{r_k-2})\\
&= (0, 0, \xi_1(\alpha(x))\otimes\xi_2(x)\otimes a_0, \dots, \xi_1(\alpha^{r_k-2}(x))\otimes \xi_2(\alpha^{r_k-3}(x))\otimes a_{r_k-3})\\
&= (0,0, \psi(\xi_1\otimes\xi_2)(x)\otimes a_0, \dots, \psi(\xi_1\otimes\xi_2)(\alpha^{r_k-3}(x))\otimes a_{r_k-3}).
\end{align*}
From a similar calculation and Proposition~\ref{tensortosection}, we see that for $\xi \in \mathcal{E}_Y^{\otimes n}=E_n$, $\pi_k(\xi)(x)$ acts on $\,\mathscr D_x^{(r_k)}$ as follows:
\begin{align*}
\lefteqn{\pi_k(\xi)(x) (a_0, a_1, \ldots, a_{r_k-1}) }  \\
&= ( 0,\ldots,0, \psi(\xi)(x)\otimes a_0, \psi(\xi)(\alpha(x))\otimes a_1,   \ldots, \psi(\xi)(\alpha^{r_k-1-n}(x)) \otimes a_{r_k-1-n}),
\end{align*}
with zeros in the first $n$ coordinates, 
or on each component $\mathscr V_x^{(i)}$, 
\begin{equation}\label{pi_k(xi)}
 \pi_k(\xi)(x)|_{\mathscr V_x^{(i)}}(a_i)=\psi(\xi)(\alpha^i(x))\otimes a_i\in \mathscr V_{\alpha^i(x)}^{(n)}\otimes \mathscr V_x^{(i)}=\mathscr V_x^{(n+i)}.  
\end{equation}
Thus for  $k=1, \dots, K$,
\[\pi_k(E_n)=0 \ \text{ whenever } n\geq r_k.\]

\end{remark}

\begin{definition}
Let $\varsigma=(\varsigma_1,\ldots,\varsigma_K)\in \Gamma(\mathscr{M}_1) \oplus \cdots \oplus \Gamma(\mathscr{M}_K)$. 
For an integer $0\leq m\leq r_K-1$, we say that  $\varsigma_k\in \Gamma(\mathscr M_k)$ {\it has an  $m$\textsuperscript{th} subdiagonal matrix representation} if at every $x\in \overline{Y_k}$, the matrix $(H_{\mathbf U}^{(k)})^{-1}(\varsigma_k(x))$ is $m$\textsuperscript{th} subdiagonal for some (hence for all) $\mathbf U\in \mathcal U^{(r_K)}$ with $x\in \mathbf U$. Here  $(H_{\mathbf U}^{(k)})^{-1}( \varsigma_k(x))$ is regarded as the  zero matrix if $m\geq r_k$.
\end{definition}

Assume that each $\varsigma_k\in\Gamma(\mathscr M_k)$ has an $m$\textsuperscript{th} lower subdiagonal  matrix representation. 
By \eqref{m-diagonal component},  for $x \in \overline{Y_k}$, the map
\[
\varsigma_k(x) : \mathscr{D}^{(r_k)}_{x} \to \mathscr{D}^{(r_k)}_{x}
\] is $m$\textsuperscript{th} lower subdiagonal in matrix representation if and only if $\varsigma_k(x)$ maps $\mathscr{V}_x^{(i)}$ in $\mathscr{D}^{(r_k)}_{x}= \mathscr{V}_x^{(0)} \oplus \mathscr{V}_x^{(1)} \oplus \cdots \oplus \mathscr{V}_x^{(r_k-1)}$ to $\mathscr{V}_x^{(m+i)}$ for $i<r_{k}-m$ and $\{0\}$ for $r_{k}-m \leq i<r_{k} $. 
Furthermore,
the linear map 
$$\varsigma_k(x)|_{\mathscr{V}^{(i)}_x}: \mathscr{V}^{(i)}_x \to \mathscr{V}^{(m+i)}_x =\mathscr{V}_{\alpha^{i}(x)}^{(m)} \otimes \mathscr{V}^{(i)}_x$$ 
is determined by a vector $\varsigma_k^{(i)}(x) \in \mathscr{V}_{\alpha^{i}(x)}^{(m)}$ such that 
\begin{equation} \label{restriction on i}
\varsigma_k(x)|_{\mathscr{V}^{(i)}_x}: v \mapsto \varsigma_k^{(i)}(x)\otimes v,\ \text{for any $v\in \mathscr{V}^{(i)}_x$}.
\end{equation}  

From the definition of $\pi_k$, we know that if $g\in C(X)$, then  $\pi_k(g)(x)$ is represented as a diagonal matrix for all $x\in \overline{Y_k}$. Also if $\xi \in \mathcal{E}_Y^{\otimes m}\subset \mathcal O(\E_Y)$ for some $1\leq m\leq r_k-1$, then   $\pi_k(\xi)(x)$ is represented as an  $m$\textsuperscript{th} lower subdiagonal matrix since it maps $\mathscr V_x^{(i)}$ into $\mathscr V_x^{(i+m)}$ by (\ref{pi_k(xi)}). Also  if  $\varsigma=(\varsigma_1,\ldots,\varsigma_K)=\pi(\xi)$ for $\xi\in \E_Y^{\otimes m}$), then we have  $\varsigma_k^{(i)}(x)=\psi(\xi)(\alpha^i(x))$ by (\ref{s_k^i}) and (\ref{restriction on i}).

\begin{lemma}\label{lemma:surjective pi_k} 
If $Y_k$ is closed, then  $\pi_k: \mathcal O(\E_Y)\to \Gamma(\mathscr M_k)$ is  surjective. In particular, $\pi_1$ is surjective.
\end{lemma}

\begin{proof} 
If $k=1$ and $r_1=1$, then $\mathscr D^{(r_1)}=\mathscr V^{(0)}=X\times \mathbb C$ is a trivial line bundle, hence $\mathscr M_1=\mathscr M^{(r_1)}|_{Y_1}=Y_1\times M_1(\mathbb C)=Y_1\times\mathbb C$. Thus $\Gamma(\mathscr M_1)=C(Y_1)$ and  $\Gamma(\mathscr M_1)=\pi_1(C(X))$ by definition of $\pi_1$. 

We assume $r_k>1$. Since the $C^*$-algebra $\Gamma(\mathscr M_k)$ is generated by the sections 
$\varsigma\in \Gamma(\mathscr M_k)$  which is an $m$\textsuperscript{th} lower subdiagonal matrix in its matrix representation for $0\leq m<r_k$, it is enough to show that  $\varsigma_k\in \pi_k(\mathcal O(\E_Y))$ for any   such section $\varsigma_k$. 
 Note that if $0\leq i\neq j <r_k$, then  
$\alpha^i(Y_k)\cap \alpha^j(Y_k)=\emptyset$ by definition of $r_k$.

Let $\varsigma_k\in \Gamma(\mathscr M_k)$ have $m$\textsuperscript{th} lower subdiagonal matrix representation. From the discussion given before this lemma, for $i=0, \dots, r_k-m-1$, there exists  $\varsigma_k^{(i)}(x) \in \mathscr{V}_{\alpha^{i}(x)}^{(m)}$ such that the map
\[\varsigma_k(x)|_{\mathscr{V}^{(i)}_x}: \mathscr{V}^{(i)}_x \to \mathscr{V}^{(m+i)}_x =\mathscr{V}_{\alpha^{i}(x)}^{(m)} \otimes \mathscr{V}^{(i)}_x\] 
is given by  $v\mapsto \varsigma_k^{(i)}(x)\otimes v$ for $v\in \mathscr{V}^{(i)}_x$.
 
 Define  $\xi\in \Gamma(\mathscr V^{(m)}|_{\cup_{i=0}^{r_k-m-1}\alpha^i(Y_k)})$ on the closed subset $\cup_{i=0}^{r_k-m-1}\alpha^i(Y_k)$ of $X$ by 
\[\xi|_{\alpha^i(Y_k)}(\alpha^i(x)):=\varsigma_k^{(i)}(x)\]
for $x\in Y_k$ and $i=0,\dots, r_k-m-1$. This $\xi$ is well defined since the sets $\alpha^i(Y_k)$ are mutually disjoint. 
It is easy to see that  
\[(\cup_{l=1}^m \alpha^{-l}(Y)) \cap (\cup_{i=0}^{r_k-m-1}\alpha^i(Y_k))=\emptyset,\] 
and then  by Tietze extension theorem for vector bundles \cite[Lemma 1.4.1]{Ati:k-theory}, $\xi$ extends to  a section $\widetilde{\xi}\in \Gamma(\mathscr V^{(m)})$ over the whole space $X$ such that for all $0 \leq i \leq r_k-m-1$,
\[\widetilde{\xi}|_{\alpha^i(Y_k)}=\xi|_{\alpha^i(Y_k)} \ \ \text{and } \ \widetilde{\xi}|_{\cup_{l=1 }^m \alpha^{-l}(Y)}=0.\]
Recall from Proposition~\ref{EandSections} that there is an isomorphism $\psi:\E^{\otimes m}\to \Gamma(\mathscr V^{(m)},\alpha^m)$ such that  the property $\widetilde{\xi}|_{\cup_{l=1}^m \alpha^{-l}(Y)}=0$ implies  $\psi^{-1}(\widetilde \xi)\in \E_Y^{\otimes m}=E_m\subset \mathcal O(\E_Y)$ (see Remark~\ref{remark:identification}).  
Now for 
any $(a_0, a_1, \ldots, a_{r_k-1})\in \mathscr D^{(r_k)}_x=\mathscr V^{(0)}_x\oplus\cdots\oplus \mathscr V^{(r_k-1)}_x$, we then have 
\begin{align*}
&\  \pi_k(\psi^{-1}(\widetilde \xi))(x)(a_0, a_1, \ldots, a_{r_k-1}) \\
=&\ ( 0,\ldots,0, \psi(\psi^{-1}(\widetilde \xi))(x)\otimes a_0,  \ldots, \psi(\psi^{-1}(\widetilde \xi))(\alpha^{r_k-m-1}(x)) \otimes a_{r_k-m-1})\\
=&\ ( 0,\ldots,0, \widetilde \xi(x)\otimes a_0,  \widetilde \xi(\alpha(x))\otimes a_1, \ldots,  \widetilde \xi(\alpha^{r_k-m-1}(x)) \otimes a_{r_k-m-1})\\
=&\ ( 0,\ldots,0, \varsigma_k^{(0)}(x)\otimes a_0,  \varsigma_k^{(1)}(x)\otimes a_1, \ldots,  \varsigma_k^{(r_k-m-1)}(x) \otimes a_{r_k-m-1})\\
=&\ \varsigma_k(x)(a_0, a_1, \ldots, a_{r_k-1}),
\end{align*}
hence $\pi_k(\psi^{-1}(\widetilde \xi))=\varsigma_k$ follows. 
\end{proof}


\section{A recursive subhomogeneous decomposition \texorpdfstring{of $\mathcal O(\E_Y)$}{}}
\label{sec:rsh construction}

In this section we explicitly describe the image of $ \mathcal{O}(\E_Y)$ under the injective $^*$-homomorphism $\pi : \mathcal{O}(\E_Y) \to \Gamma(\mathscr M_1) \oplus \dots\oplus \Gamma(\mathscr M_K)$.  In particular, we construct a recursive subhomogeneous decomposition of $\pi(\mathcal{O}(\E_Y)) \cong \mathcal{O}(\E_Y)$. As in the case of $\mathrm C^*$-algebras of minimal homeomorphisms, we investigate the behaviour of boundary points. However, the presence of a twist (that is, a non-trivial line bundle) leads to further difficulties since we need to consider the local trivialization of the underlying endomorphism bundles when we are tracking the travel of boundary points through the Rokhlin towers.

Note that the sets $Y_k$ are in general not closed.
If $x \in \overline{Y_k}$, by continuity of $\alpha$ and the fact that $Y$ is closed, we have $\alpha^{r_k} (x) \in Y$. Whereas, if $x \in \overline{Y_k} \setminus Y_k$, then this will not be the first return time of $x$ so that $x\in Y_{t_1}$ for some $t_1 < k$, and hence $\alpha^{r_{t_1}} (x) \in   Y_{t_2}$ for some $t_2 < k$. It follows that $x \in \overline{Y_k} \cap Y_{t_1} \cap \alpha^{-{r_{t_1}}}(Y_{t_2})$.   If $r_{t_1} + r_{t_2} < r_k$, then there is some $t_3 < k$ such that $\alpha^{r_{t_1} + r_{t_2}}(x)\in Y_{t_3}$. Continuing in this manner, we get
\[
t_1, t_2, \ldots , t_{m(x,k)} < k
\]
such that
\[
r_{t_1} + r_{t_2} + \cdots + r_{t_{m(x,k)}} =   r_{ k}
\] 
and
\[
x \in \overline{Y_k} \cap Y_{t_1} \cap \alpha^{-{r_{t_1}}}(Y_{t_2}) \cap \cdots
\cap \alpha^{-(r_{t_1} + r_{t_2} + \cdots + r_{t_{m(x,k)-1}})}(Y_{t_{m(x,k)}})
\] 
Define $\mu_{x,k}$  to be the ordered multi-set
\[
\mu_{x,k} := \{t_1, t_2, \ldots , t_{m(x,k)} \},
\] 
and for $1 \leq s \leq m (x, k)$ let $ \mu_{x,k}(s)=t_s$ and $|\mu_{x,k}|= m (x, k)$.
For $2 \leq k \leq K$, and $x \in \overline{Y_k}$, let
\[
R_{x,k,0}:=0,\quad
R_{x,k,s}:=r_{t_1} + r_{t_2} + \cdots + r_{t_s}=\sum_{t=1}^s r_{\mu_{x,k}(t)},\ 1\leq s \leq |\mu_{x,k}|.
\]
Note that the vector space 
\begin{align*}
\mathscr{D}^{(r_k)}_x &= \mathscr{V}^{(0)}_x \oplus\cdots\oplus \mathscr{V}^{(r_{t_1}-1)}_x \oplus \mathscr{V}^{(r_{t_1})}_x \oplus \cdots \oplus \mathscr{V}^{(r_{t_1}+r_{t_2}-1)}_x\oplus\cdots\oplus \mathscr{V}^{(r_k-1)}_x\\
&=\mathscr{D}_x^{(r_{t_1})} \oplus \mathscr{D}_x^{(R_{x,k,1}, r_{t_2})} \oplus \mathscr{D}_x^{(R_{x,k,2}, r_{t_3})} \oplus \cdots \oplus \mathscr{D}_x^{(R_{x,k,m(x,k)-1},r_{t_{m(x,k)}})}
\end{align*}
is the direct sum of $\mathscr{D}^{(R_{x,k,s-1}, r_{t_s})}$, $s=1,2,\ldots,m(x,k)$, where $\mathscr D^{(0,r_{t_1})}:=\mathscr D^{(r_{t_1})}$. We call \[\mathscr{D}_x^{(R_{x,k,s-1}, r_{t_s})}=\mathscr V_x^{(R_{x,k,s-1})}\oplus\mathscr V_x^{(R_{x,k,s-1}+1)}\oplus\cdots\oplus \mathscr V_x^{(R_{x,k,s-1}+r_{t_s}-1)}\] the {\it $s$\textsuperscript{th} component } of $\mathscr D_x^{(r_k)}$ for $s=1, \dots, m(x,k)$, which  
can also be written as follows: 
\begin{align*}
\lefteqn{\mathscr{D}^{(R_{x,k,s-1}, r_{t_s})}_x} \\
& =\ \mathscr V_x^{(R_{x,k,s-1})}\oplus \mathscr V_x^{(R_{x,k,s-1}+1)}\oplus \cdots\oplus \mathscr V_x^{(R_{x,k,s}-1)}\\
& =\ \mathscr V_x^{(R_{x,k,s-1})}\oplus \mathscr V_x^{(R_{x,k,s-1}+1)}\oplus \cdots\oplus \mathscr V_x^{(R_{x,k,s-1}+r_{t_s}-1)}\\
& =\ \big(\mathscr V_{\alpha^{R_{x,k,s-1}}(x)}^{(0)}\otimes \mathscr V_x^{(R_{x,k,s-1})}\big)\oplus\cdots\oplus \big(\mathscr V_{\alpha^{R_{x,k,s-1}}(x)}^{(r_{t_s}-1)}\otimes \mathscr V_x^{(R_{x,k,s-1})}\big)\\
& =\ \big( \mathscr V_{\alpha^{R_{x,k,s-1}}(x)}^{(0)}\oplus\cdots\oplus \mathscr V_{\alpha^{R_{x,k,s-1}}(x)}^{(r_{t_s}-1)}\big)\otimes  \mathscr V_x^{(R_{x,k,s-1})}\\
& =\ \mathscr D_{\alpha^{R_{x,k,s-1}}(x)}^{(r_{t_s})}\otimes  \mathscr V_x^{(R_{x,k,s-1})}.
\end{align*} 

The next definition is based on \cite[Lemma 4]{QLin:Ay}.
\begin{definition}\label{bd property}
For each $\varsigma=(\varsigma_1,\ldots,\varsigma_K)\in \Gamma(\mathscr{M}_1) \oplus \cdots \oplus \Gamma(\mathscr{M}_K)$,  we say that $\varsigma$ has the \emph{boundary decomposition property} if it satisfies the following: for any $x\in \overline{Y_k}\setminus Y_k$, $k=1,\dots,K$ with 
\[
x \in  \overline{Y_k} \cap Y_{t_1} \cap \alpha^{-{r_{t_1}}}(Y_{t_2}) \cap \cdots
\cap \alpha^{-(r_{t_1} + r_{t_2} + \cdots + r_{t_{m(x,k)-1}})}(Y_{t_{m(x,k)}})
\] 
and $r_{t_1} + r_{t_2} + \cdots + r_{t_{m(x,k)}} =   r_{ k}$, 
the linear map $\varsigma_k(x)\in \mathscr M_k|_x=End(\mathscr D_x^{(r_k)})$ acts on the $s$\textsuperscript{th} component 
\[\mathscr{D}^{(R_{x,k,s-1}, r_{t_s})}_x=\ \mathscr D_{\alpha^{R_{x,k,s-1}}(x)}^{(r_{t_s})}\otimes  \mathscr V_x^{(R_{x,k,s-1})}\]
of $\mathscr{D}^{(r_{k})}_x$ as  
\begin{equation}\label{s-component}
\varsigma_k(x)|_{\mathscr{D}^{(R_{x,k,s-1}, r_{t_s})}_x} = \varsigma_{t_s}(\alpha^{R_{x,k,s-1}}(x)) \otimes id_{\mathscr{V}^{(R_{x,k,s-1})}_x},      
\end{equation}
or equivalently, $\varsigma_k(x)$ has its matrix representation of the following form 
\[
\widetilde{\varsigma}_k(x) = \begin{bmatrix}
\widetilde{\varsigma}_{t_1}(x) & \\
 & \widetilde{\varsigma}_{t_2}( \alpha^{R_{x,k,1}}(x)) & \\
&         & \ddots\\
&         &    & \widetilde{\varsigma}_{t_{ m(x,k)}}(\alpha^{R_{x,k,m(x,k)-1}}(x))   
\end{bmatrix}
\]
with respect to the local trivialization induced from $\{ (U,h_U)\mid U\in \mathcal U^{(1)}=\mathcal U\}$ of  $\mathscr{V}$.	
\end{definition}
Next we show that the image of $\pi$ can be determined by the boundary decomposition property.

\begin{lemma}\label{lemma:bdp_ran(pi)} If $\varsigma\in \pi(\mathcal O(\E_Y))$, then $\varsigma$   has the boundary decomposition property.
\end{lemma}

\begin{proof} 
Since it is not hard to see that the boundary decomposition property is preserved under addition, multiplication, involution, and taking limits,  we only need  to show the assertion when $\varsigma=\pi(g)$ for $g\in C(X)$ or  $\varsigma=\pi(\xi)$ for  $\xi\in \E_Y$.

First let $\varsigma=(\varsigma_1,\ldots,\varsigma_K)=\pi(g)\in \Gamma(\mathscr{M}_1) \oplus \cdots \oplus \Gamma(\mathscr{M}_K)$  for  some $g\in C(X)$ and let $x\in \overline{Y_k}\setminus Y_k$ be a boundary point of $Y_k$ as in Definition~\ref{bd property}.  We then have 
\begin{align*}
	\varsigma_k(x) &=\pi_k(g)(x)= \diag(g(x), g(\alpha(x)), \ldots, g(\alpha^{r_k-1}(x)) ):\mathscr{D}^{(r_k)}_{x} \to \mathscr{D}^{(r_k)}_{x}
\end{align*}
and the restriction  $\varsigma_k(x)|_{\mathscr{V}_x^{(i)}} : \mathscr{V}_x^{(i)} \to \mathscr{V}_x^{(i)}$ is a multiplication by a number which we write as $\varsigma_k^{(i)}(x)$ as observed in (\ref{restriction on i}) with $m=0$. 
On the $s$\textsuperscript{th} component 
\[
\mathscr{D}^{(R_{x,k,s-1}, r_{t_s})}_x= \big( \mathscr{V}_{\alpha^{R_{x,k,s-1}}(x)}^{(0)} \oplus   \mathscr{V}_{\alpha^{R_{x,k,s-1}}(x)}^{(1)} \oplus \cdots \oplus \mathscr{V}_{\alpha^{R_{x,k,s-1}}(x)}^{(r_{t_{s}}-1)}\big)\otimes \mathscr{V}_x^{(R_{x,k,s-1})}
\]
of $\mathscr{D}^{(r_{k})}_x$, 
the map $\varsigma_k(x) =\pi_k(g)(x)$ acts as 
\[\diag(g(\alpha^{R_{x,k,s-1}}(x)), g(\alpha^{R_{x,k,s-1}+1}(x)), \ldots, g(\alpha^{R_{x,k,s-1}+r_{t_s}-1}(x)) ),\]
whereas $\varsigma_{t_s} \in \Gamma(\mathscr{M}_{t_s})$ acts on $\alpha^{R_{x,k,s-1}}(x)\in Y_{t_s}$ as 
\begin{align*}
\lefteqn{\varsigma_{t_s}(\alpha^{R_{x,k,s-1}}(x))} \\ 
&= \pi_{t_s}(g)(\alpha^{R_{x,k,s-1}}(x))\\
&= \diag(g(\alpha^{R_{x,k,s-1}}(x)), g(\alpha(\alpha^{R_{x,k,s-1}}(x))), \ldots, g(\alpha^{r_{t_s }-1}(\alpha^{R_{x,k,s-1}}(x))) )	
\end{align*}
 on the vector space 
 \begin{align*}
\lefteqn{\mathscr{D}_{\alpha^{R_{x,k,s-1}}(x)}^{(r_{t_s})}} \\
&= \mathscr{V}_{\alpha^{R_{x,k,s-1}}(x)}^{(0)} \oplus \mathscr{V}_{\alpha^{R_{x,k,s-1}}(x)}^{(1)} \oplus \mathscr{V}_{\alpha^{R_{x,k,s-1}}(x)}^{(2)} \oplus \cdots \oplus \mathscr{V}_{\alpha^{R_{x,k,s-1}}(x)}^{(r_{t_s }-1)} . \end{align*}
Therefore on $s$\textsuperscript{th} component $\mathscr{D}^{(R_{x,k,s-1}, r_{t_s})}_x=\mathscr{D}_{\alpha^{R_{x,k,s-1}}(x)}^{(r_{t_s})}\otimes \mathscr{V}_x^{(R_{x,k,s-1})}$,
\[
\varsigma_k(x) = \varsigma_{t_s}(\alpha^{R_{x,k,s-1}}(x)) \otimes \id_{\mathscr{V}^{(R_{x,k,s-1})}_x}, 
\]
from which we have
\[
\varsigma_k^{(R_{x,k,s-1}+i)}(x) = \varsigma_{t_s}^{(i)}(\alpha^{R_{x,k,s-1}}(x)).
\]  

Now let $\varsigma=\varsigma_1\oplus\cdots\oplus\varsigma_K=\pi_1(\xi)\oplus\cdots\oplus \pi_K(\xi)$ for some $\xi\in \E_Y$. Then $\pi_k(\xi)(x)$ acts on the $s$\textsuperscript{th} component of $\mathscr D_x^{(r_k)}$ as follows. If $(a_0,a_1,\dots, a_{r_k-1})\in \mathscr D_x^{(r_k)}$, then its $s$\textsuperscript{th} component is equal to \[(a_{R_{x,k,s-1}}, \dots, a_{R_{x,k,s}-1})\in \mathscr D_x^{(R_{x,k,s-1},r_{t_s})}\] and 
\begin{align*}
\lefteqn{\varsigma_k(x)|_{\mathscr D_x^{(R_{x,k,s-1},r_{t_s})}}(a_{R_{x,k,s-1}}, \dots, a_{R_{x,k,s}-1}) } \\
&= \pi_k(\xi)(x)|_{\mathscr D_x^{(R_{x,k,s-1},r_{t_s})}} (a_{R_{x,k,s-1}}, \dots, a_{R_{x,k,s}-1})  \\
&=  (0, \xi(\alpha^{R_{x,k,s-1}}(x))\otimes a_{R_{x,k,s-1}}, \dots, \xi(\alpha^{R_{x,k,s}-2}(x))\otimes a_{R_{x,k,s}-2} )\\
&= \pi_{t_s}(\xi)(\alpha^{R_{x,k,s-1}}(x)) (a_{R_{x,k,s-1}}, \dots, a_{R_{x,k,s}-1})\\
&=  \varsigma_{t_s}(x)(\alpha^{R_{x,k,s-1}}(x)) (a_{R_{x,k,s-1}}, \dots, a_{R_{x,k,s}-1}), 
\end{align*}   
which proves the lemma. 
\end{proof}

Recall from (\ref{restriction on i}) that if $\varsigma_k\in \Gamma(\mathscr M_k)$ has an $m^{th}$ lower subdiagonal matrix representation, then we have a map \[x\mapsto \varsigma_k^{(i)}(x):\overline{Y_k}\to \mathscr V_{\alpha^i(x)}^{(m)}\] obtained by restricting the linear map $\varsigma_k(x):\mathscr D^{(r_k)}_x\to \mathscr D^{(r_k)}_x$ onto  $\mathscr V_x^{(i)}$, a direct summand of $\mathscr D^{(r_k)}_x$.

\begin{lemma}\label{lemma:BC}
Let  $\varsigma=(\varsigma_1,\ldots,\varsigma_K)\in \Gamma(\mathscr{M}_1) \oplus \cdots \oplus \Gamma(\mathscr{M}_K)$ have the boundary decomposition property, and let $0\leq m< r_K$. If  $\varsigma_k$ has an $m$\textsuperscript{th} subdiagonal matrix representation for all $k$, then  there is a section $G\in \Gamma(\mathscr{V}^{(m)})$  of $\mathscr{V}^{(m)}$ such that 
\[
G\big|_{\alpha^i(\overline{Y_k})}= \varsigma_{k}^{(i)} \circ \alpha^{-i} \big|_{\alpha^i(\overline{Y_k})},  
\]
for all $0\leq i\leq r_{k}-1,$ and 
vanishes on $\alpha^{-m}(Y)\cup \alpha^{-m+1}(Y)\cup \cdots \cup \alpha^{-1}(Y)$ provided $1\leq m< r_k$.
\end{lemma}

\begin{proof} To construct a continuous section $G:X\to \mathscr V^{(m)}$ satisfying the properties stated in this lemma, we first note $X=\cup_{k=1}^K\cup_{i=0}^{r_k-1}\alpha^i(\overline{Y_k})$ and define $G$ on each $\alpha^i(\overline{Y_k})$ in such a way that the values of $G$ match up at any possible overlaps $\alpha^j(\overline{Y_l})\cap \alpha^i(\overline{Y_k})$. 

Let $x\in \overline{Y_k}\setminus Y_k$ be a point as in Definition~\ref{bd property}. Since $\varsigma$ has the boundary decomposition property, on each $s$\textsuperscript{th} component $\mathscr{D}^{(R_{x,k,s-1}, r_{t_s})}_x$ of $\mathscr{D}^{(r_{k})}_x$, 
\begin{equation}\label{eq:s-comp-decom}
\varsigma_k(x)|_{\mathscr{D}^{(R_{x,k,s-1}, r_{t_s})}_x} = \varsigma_{t_s}(\alpha^{R_{x,k,s-1}}(x)) \otimes id_{\mathscr{V}^{(R_{x,k,s-1})}_x}, 	
\end{equation} 
for $s=2,\dots, m(x,k)$. 
Also by assumption,
$ \varsigma_k(x) : \mathscr{D}^{(r_{k})}_x \to \mathscr{D}^{(r_{k})}_x $
has an $m$\textsuperscript{th} subdiagonal matrix representation, which means that $\varsigma_k(x)$ maps  $\mathscr V_x^{(i)}$ into  $\mathscr V_x^{(i+m)}=\mathscr V_{\alpha^i(x)}^{(m)}\otimes\mathscr V_x^{(i)}$. Hence,  
there exist elements $\varsigma_k^{(i)}(x) \in \mathscr{V}_{\alpha^{i}(x)}^{(m)}$ such that
\[
\varsigma_k(x)|_{\mathscr V_x^{(i)}} :   v \mapsto \varsigma_{k}^{(i)}(x)\otimes v,\quad   v \in \mathscr{V}^{(i)}_x, 
\]
where $\varsigma_k^{(i)}(x)=0$ for $i \geq r_k-m$. 
Equation (\ref{eq:s-comp-decom}) implies that for $i$ with $R_{x,k,s-1} \leq i < R_{x,k,s}$,  
\[
\varsigma_{k}^{(i)}(x)=\varsigma_{t_s}^{(i-R_{x,k,s-1})}(\alpha^{R_{x,k,s-1}}(x)).
\]

Now we define $G:X\to \mathscr V^{(m)}$   by
\[
G\big|_{\alpha^i(\overline{Y_k})}= \varsigma_{k}^{(i)} \circ \alpha^{-i} \big|_{\alpha^i(\overline{Y_k})},
\] 
for $i=0,1,2,\ldots, r_{k}-1$. 
The section $G$ is well defined if $x\in \alpha^{i}(\overline{Y_k}) \cap \alpha^{j}(\overline{Y_l})$ implies
\begin{equation}\label{eq:G-ij}
	\varsigma_{k}^{(i)} \circ \alpha^{-i}(x) =\varsigma_{l}^{(j)} \circ \alpha^{-j}(x).
\end{equation}
By (\ref{X:union}), it is enough to show that if $x\in \alpha^{i}(\overline{Y_k}) \cap \alpha^{j}(\overline{Y_l}) \cap \alpha^p(Y_q)$, then (\ref{eq:G-ij}) holds.
If $0<p-i<r_q$, then we would have $\overline{Y_k} \cap \alpha^{p-i}(Y_q) \neq \emptyset$, contradicting the fact that $Y \cap \alpha^{p-i}(Y_q) = \emptyset$. Thus 
\begin{equation}\label{ip}
\alpha^{i}(\overline{Y_k}) \cap \alpha^{p}(Y_q)\neq \emptyset \ \text{ implies }\ i\geq p
\end{equation} 
whenever $0\leq i<r_k$ and  $0\leq p<r_q$.

If $x=\alpha^i(y) \in \alpha^{i}(\overline{Y_k}) \cap \alpha^{p}(Y_q)$, for $p\leq i<r_k$ and 
$q$, which implies that  $y \in \overline{Y_k} \setminus Y_k$, 
then 
\[
\alpha^{i-p}(y)=\alpha^{-p}(x) \in Y_q.
\]
 Since $0\leq i-p <r_k$, there is an $n\geq 0$ such that
 \[
 r_{t_1} + r_{t_2} + \cdots + r_{t_n}  \leq i-p < r_{t_1} + r_{t_2} + \cdots + r_{t_{n+1}}
 \] with $r_{t_0}:=0$.
Since
\[
\alpha^{r_{t_1} + r_{t_2} + \cdots + r_{t_n}}(y) \in   Y_{t_{n+1}},\] 
the identity
\[
\alpha^{i-p-(r_{t_1} + r_{t_2} + \cdots + r_{t_n} )}(\alpha^{r_{t_1} + r_{t_2} + \cdots + r_{t_n} }(y))=\alpha^{i-p}(y) \in  Y
\]
implies that, since $ 0 \leq i-p-(r_{t_1} + r_{t_2} + \cdots + r_{t_n} ) < r_{t_{n+1}}$, 
\[
i-p = r_{t_1} + r_{t_2} + \cdots + r_{t_n}\ \text{ and }\ q=t_{n+1}.
\]
Thus we have 
\[
i = (r_{t_1} + r_{t_2} + \cdots + r_{t_n}) + p = R_{y,k,n} + p,
\]
and, by assumption,
\begin{align*}
\varsigma_{k}^{(i)} \circ \alpha^{-i}(x) & = \varsigma_{k}^{(i)}(y) \\
& =\varsigma_{t_{n+1}}^{(p)}( \alpha^{R_{y,k,n}}(y)) \\
& =\varsigma_{t_{n+1}}^{(p)}( \alpha^{i-p}(y)) \\
& =\varsigma_{q}^{(p)}( \alpha^{i-p}(y)) \\
& =\varsigma_{q}^{(p)} \circ \alpha^{-p}(x).	
\end{align*}
Similarly
\[
\varsigma_{l}^{(j)} \circ \alpha^{-j}(x)= \varsigma_{q}^{(p)} \circ \alpha^{-p}(x)
\]
can be obtained, and we see that $G$ is well defined.

Finally, if $x \in \alpha^{-j}(Y_q)$ for $j=1,\ldots,m$, then $x \in \alpha^i(Y_k) \cap \alpha^{-j}(Y_q)$ for some $0\leq i<r_k$, $1\leq k\leq K$, hence  $\alpha^j(x) \in \alpha^{i+j}(Y_k) \cap Y_q \neq \emptyset$. Then $i+j \geq r_{k}$, and $i \geq r_{k} -j \geq r_{k} -m$.
Therefore 
\[
G(x) = \varsigma_k^{(i)} \circ \alpha^{-i}(x)=0,
\]
since $\varsigma_k^{(i)}=0$ for $i \geq r_{k} -m$.
\end{proof}

\begin{corollary}\label{lemma:functionAh}
Let $\varsigma=(\varsigma_1,\ldots,\varsigma_K)\in \Gamma(\mathscr{M}_1) \oplus \cdots \oplus \Gamma(\mathscr{M}_K)$. Then  $\varsigma \in \pi(C(X))$ if and only if  $\varsigma$ has the boundary decomposition property, and $\varsigma_k(x) \in \Gamma(\mathscr{M}_k|_x)$ has a   diagonal matrix representation  for all  $x\in \overline{Y_k}$, $k=1,\dots, K$.
\end{corollary}

\begin{proof} 
Assume that $\varsigma$ has the boundary decomposition property and every $\varsigma_k(x)$ has a diagonal matrix representation. Then by Lemma \ref{lemma:BC} with $m=0$, there is a function $g\in C(X)=\Gamma(\mathscr V^{(0)})$ given by
\[
g\big|_{\alpha^i(\overline{Y_k})}= \varsigma_{k}^{(i)} \circ \alpha^{-i} \big|_{\alpha^i(\overline{Y_k})},\quad i=0,1,\dots, r_{k}-1.
\]	
Thus $\varsigma_{k}^{(i)}(x) = g(\alpha^i(x))$ for $x \in \overline{Y_k}$, which means that $\varsigma_k= \pi_k(g)$ for all $k$.
\end{proof}

\begin{corollary} \label{cor:boundarydecomp}
Let $\varsigma=(\varsigma_1,\ldots,\varsigma_K)\in \Gamma(\mathscr{M}_1) \oplus \cdots \oplus \Gamma(\mathscr{M}_K)$ and $m$ be a positive integer. Then  $\varsigma \in \pi(\mathcal{E}_Y^{\otimes m})$ if and only if $\varsigma$ satisfies the boundary decomposition property and has an $m$\textsuperscript{th} lower subdiagonal matrix representation. 
\end{corollary}

\begin{proof} 
Assume that $\varsigma$ satisfies the boundary decomposition property and has an $m$\textsuperscript{th} (lower) subdiagonal matrix representation. 
Then 	
$\varsigma_k(x)$ maps $\mathscr{V}_x^{(i)}$ in $\mathscr{D}^{(r_k)}_{x}= \mathscr{V}_x^{(0)} \oplus \mathscr{V}_x^{(1)} \oplus \mathscr{V}_x^{(2)} \oplus \cdots \oplus \mathscr{V}_x^{(r_k-1)}$ to $\mathscr{V}_x^{(m+i)}$ for $i<r_{k}-m$ and $\{0\}$ for $r_{k}-m \leq i<r_{k} $. 
Again by Lemma \ref{lemma:BC}, there is a $G \in \Gamma(\mathscr{V}^{(m)}, \alpha^m )$ such that 
\[
G\big|_{\alpha^i(\overline{Y_k})}= \varsigma_{k}^{(i)} \circ \alpha^{-i} \big|_{\alpha^i(\overline{Y_k})},\quad i=0,1,2,\ldots, r_{k}-m-1,
\]	and 
vanishes on $\alpha^{-m}(Y) \cup\alpha^{-m+1}(Y) \cup \cdots \cup\alpha^{-1}(Y)$. By Proposition \ref{EandSections}, there is $\xi \in \mathcal{E}_Y^{\otimes m}$ with $\psi(\xi)=G$. From the definition of $G$ we see that $\varsigma=\pi(\xi)$.
\end{proof}

Since $\Gamma(\mathscr{M}_1) \oplus \cdots \oplus \Gamma(\mathscr{M}_K)$ is generated   (as a $\mathrm C^*$-algebra) by  elements that have $m$\textsuperscript{th} (lower) subdiagonal matrix representation, we obtain the following theorem.

\begin{theorem}\label{thm:bdp}
For $\varsigma=(\varsigma_1,\ldots,\varsigma_K)\in \Gamma(\mathscr{M}_1) \oplus \cdots \oplus \Gamma(\mathscr{M}_K)$, $\varsigma$  satisfies the boundary decomposition property   if and only if $\varsigma \in \pi(\mathcal{O}(\mathcal{E}_Y))$.	
\end{theorem}


\subsubsection{An RSH structure for $\mathcal O(\E_Y)$ }

Let $x\in \overline{Y_2}\setminus Y_2=\overline{Y_2}\cap Y_1$ be a point as in Definition~\ref{bd property}.  Then $r_{t_s}=r_1$ for all $s=1,\dots, m:=m(x,2)$ and thus 
$R_{x,2,s}=r_{t_1}+\cdots+r_{t_s}=sr_1$ for every $s$. Note that $m(x,2)=m$ for every $x\in \overline{Y_2}\setminus Y_2$.

Let $B_1:=\Gamma(\mathscr M_1)$ and let $\rho_1:\Gamma(\mathscr M_2)\to \Gamma(\mathscr M_2|_{\overline{Y_2}\setminus Y_2})$ be the restriction map. By Lemma~\ref{lemma:surjective pi_k}, the map $\pi_1 : \mathcal{O}(\E_Y) \to \Gamma(\mathscr{M}_1)$ is surjective. Thus, for any $\varsigma_1 \ \in \Gamma(\mathscr M_1)$, there exists $\xi \in \mathcal{O}(\E_Y)$ such that $\pi_1(\xi) = \varsigma_1$. Define 
\[\varphi_1:\Gamma(\mathscr M_1)\to \Gamma(\mathscr M_2|_{\overline{Y_2}\setminus Y_2})\] 
by 
\[
\varphi_1(\varsigma_1)(x) :=\pi_2(\xi)(x).
\]
We claim that $\varphi_1$ is a well-defined $^*$-homomorphism.  Suppose that  $\pi_1(\xi_1)=\pi_1(\xi_2)=\varsigma_1$. By Theorem~\ref{thm:bdp}, $\varsigma=(\varsigma_1, \dots, \varsigma_K):=\pi(\xi_1)\in \Gamma(\mathscr M_1)\oplus\cdots\oplus\Gamma(\mathscr M_K)$  satisfies the boundary decomposition property. Then, using (\ref{s-component}), we have that  
\begin{align*}
\pi_2(\xi_1)(x)|_{\mathscr D_x^{(R_{x,2,s-1},r_{1})}}&\ = \varsigma_2(x)|_{\mathscr D_x^{(R_{x,2,s-1},r_{1})}}\\
&\ = \varsigma_1(\alpha^{R_{x,2,s-1}}(x))\otimes id_{\mathscr V_x^{(R_{x,2,s-1})}}\\
&\ =\pi_1(\xi_1)(\alpha^{R_{x,2,s-1}}(x))\otimes id_{\mathscr V_x^{(R_{x,2,s-1})}}\\
&\ =\pi_1(\xi_2)(\alpha^{R_{x,2,s-1}}(x))\otimes id_{\mathscr V_x^{(R_{x,2,s-1})}}\\
& \ =\pi_2(\xi_2)(x)|_{\mathscr D_x^{(R_{x,2,s-1},r_{1})}}.
\end{align*} 
It follows $\varphi_1$ is well defined, and obviously it is a $^*$-homomorphism, which proves the claim. Let 
\begin{align*}
B_2:=&\ \Gamma(\mathscr M_1)\oplus_{\Gamma(\mathscr M_2|_{\overline{Y_2} \setminus Y_2})} \Gamma(\mathscr M_2)\\
=&\ \{(\varsigma_1,\varsigma_2)\in \Gamma(\mathscr M_1)\oplus \Gamma(\mathscr M_2)\mid \varphi_1(\varsigma_1)=\rho_1(\varsigma_2)\}.   
\end{align*}
Then $B_2$ is an RSH algebra.

As in \cite[Theorem 11.3.19]{GioKerPhi:CRM}, one can continue this process to form an RSH subalgebra $B_k$ of $\Gamma(\mathscr M_1)\oplus\cdots\oplus \Gamma(\mathscr M_k)$ for all $k=3,\dots, K$ simultaneously showing that at each step $k$, every element in $B_k$ satisfies the boundary decomposition property (only) up to $k$.

\begin{theorem}\label{thm:rsh} 
For every $k = 1, 2,\ldots,K$, $B_k=\oplus_{i=1}^k\pi_i(\mathcal O(\E_Y))$ is a recursive subhomogeneous $\mathrm C^*$-subalgebra of $\oplus_{i=1}^k \Gamma(\mathscr M_i)$ such that  $B_1=\Gamma(\mathscr M_1)$, and for $k=2, \dots, K$, 
 \[B_k=B_{k-1}\oplus_{\Gamma(\mathscr M_{k}|_{\overline{Y_k}\setminus Y_k})} \Gamma(\mathscr M_k).\]  
\end{theorem}

 \begin{proof}
We prove the theorem by induction on $k$. First note that by \cite[Proposition 1.7]{Phillips:recsub} , $B_1 := \Gamma(\mathscr M_1)$  is a recursive subhomogeneous algebra and $\pi_1:\mathcal O(\E_Y)\to  \Gamma(\mathscr M_1)$ is surjective by Lemma~\ref{lemma:surjective pi_k}. 
 
Assume that for $2\leq k\leq K-1$, the $^*$-homomorphism \[\oplus_{i=1}^{k-1} \pi_i : \mathcal O(\E_Y) \to B_{k-1} \subset \oplus_{i=1}^{k-1}\Gamma(\mathscr M_i)\] is  surjective onto the recursive subhomogeneous algebra $B_{k-1}$. 
We now construct $B_k$.  For $(b_1,b_2,\ldots, b_{k-1}) \in B_{k-1}$, since $\oplus_{i=1}^{k-1} \pi_i (\mathcal O(\E_Y))=  
B_{k-1}$,  there is $\xi \in \mathcal O(\E_Y)$ with $\oplus_{i=1}^{k-1} \pi_i(\xi)=(b_1,b_2,\ldots, b_{k-1})$. Define
\[\varphi_{k-1}(b_1,b_2,\ldots, b_{k-1}) :=\pi_k(\xi)|_{\overline{Y_k} \setminus Y_k}.
\] 
By Theorem \ref{thm:bdp}, 
$\pi_k(\xi)|_{\overline{Y_k}\setminus Y_k}$  is determined by $ \pi_i(\xi) = b_i$, $1\leq i < k$, namely if $\oplus_{i=1}^{k-1} \pi_i(\xi)=(b_1, \dots, b_{k-1})=\oplus_{i=1}^{k-1} \pi_i(\xi')$ for $\xi$, $\xi'\in \mathcal O(\E_Y)$, then $\pi_k(\xi)|_{\overline{Y_k}\setminus Y_k}=\pi_k(\xi')|_{\overline{Y_k}\setminus Y_k}$. 
Hence  
\[
\varphi_{k-1} : B_{k-1} \to \Gamma(\mathscr M_k |_{\overline{Y_k} \setminus Y_k})
\]
is a well-defined $^*$-homomorphism.

Define $B_k$ to be 
\[
B_k:=B_{k-1}\oplus_{\Gamma(\mathscr M_k|_{\overline{Y_k}\setminus Y_k})}\Gamma(\mathscr M_k)\subset \oplus_{i=1}^k\Gamma(\mathscr M_i),
\] 
where $\rho_{k-1}: \Gamma(\mathscr M_k)\to \Gamma(\mathscr M_k|_{\overline{Y_k}\setminus Y_k})$ is the restriction map. 
Then $\oplus_{i=1}^k \pi_i(\mathcal O(\E_Y))$ is a $\mathrm C^*$-subalgebra of the recursive subhomogeneous algebra $B_k$.

Now we prove that $\oplus_{i=1}^{k} \pi_i : \mathcal O(\E_Y) \to B_k$ is surjective.
Let 
\[
(b_1,b_2,\ldots,b_k) \in B_k=B_{k-1}\oplus_{\Gamma(\mathscr M_k|_{\overline{Y_k}\setminus Y_k})}\Gamma(\mathscr M_k).
\] 
By the induction hypothesis,  there is $\xi \in \mathcal O(\E_Y)$ such that  
\[\oplus_{i=1}^{k-1}\pi_i(\xi)=(b_1,b_2,\ldots,b_{k-1}).
\] 
Consider the element 
\[
(b_1,b_2,\ldots,b_k) - \oplus_{i=1}^{k}\pi_i(\xi) = (0,0,\ldots, 0 , b_k - \pi_k(\xi)).
\] 
The last coordinate $b_k-\pi_k(\xi)$ is not necessarily $0$ over the whole $\overline{Y_k}$ but equal to $0$ on $\overline{Y_k}\setminus Y_k$ since $\pi_k(\xi)|_{\overline{Y_k} \setminus Y_k} =\varphi_{k-1}(b_1,b_2,\ldots, b_{k-1})=  \rho_{k-1}(b_k)$. 
It is thus enough to claim that for any $\varsigma \in \Gamma(\mathscr M_k)$ with $\varsigma|_{\overline{Y_k} \setminus Y_k}=0$, there is $\eta \in \mathcal O(\E_Y)$ with \[\pi_i(\eta)=0\ \text{for } i=1,2,\ldots,k-1, \ \text{ and } \pi_k(\eta)= \varsigma.\] 
In fact, the claim implies that for $\varsigma:=b_k - \pi_k(\xi) \in \Gamma(\mathscr M_k)$ there is such an $\eta \in \mathcal O(\E_Y)$ and then $(b_1,b_2,\ldots,b_k)=\oplus_{i=1}^{k}\pi_i(\xi+\eta)$ follows.

To prove the claim,  let $\varsigma \in \Gamma(\mathscr M_k)$ vanishes on $\overline{Y_k} \setminus Y_k$. We may assume that $\varsigma $ is $m$\textsuperscript{th} lower subdiagonal in matrix representation, that is, 
\[
\widetilde{\varsigma}_U(x) =
\begin{matrix}
\hfill m\left\{\vphantom{\begin{matrix} 0 \\ \vdots \\ 0 \end{matrix}}\right.\\[7mm] 
r_k-m\left\{\vphantom{\begin{matrix} \widetilde{s}_k^{(0)}(x) \\0 \\ \vdots \\ 0 \end{matrix}}\right.
\end{matrix}%
\begin{bmatrix}
0 & 0  & \cdots & 0 & 0 & \cdots & 0 & 0 \\
\vdots & \vdots &   & \vdots &   \vdots & &\vdots & \vdots  \\
0 & 0  & \cdots & 0 & 0 & \cdots & 0 & 0 \\
\widetilde{\varsigma}_U^{(0)}(x)  & 0 & \cdots & 0 & 0 & \cdots & 0 &  0  \\
0 &\widetilde{\varsigma}_U^{(1)}(x)   & \cdots & 0& 0 & \cdots & 0 & 0  \\
\vdots  &\vdots   & \ddots & \vdots & \vdots & \cdots & \vdots & \vdots  \\
0 & 0  & \cdots & \widetilde{\varsigma}_U^{(r_k-1-m)}(x)& 0 & \cdots & 0  & 0  
\end{bmatrix},
\]
where $U\in \mathcal U$ is any open set with $x\in U$. Then for $i=0, \dots, r_k-1-m$, there exists  $\varsigma^{(i)}(x) \in \mathscr{V}_{\alpha^{i}(x)}^{(m)}$ such that the map
\[\varsigma(x)|_{\mathscr{V}^{(i)}_x}: \mathscr{V}^{(i)}_x \to \mathscr{V}^{(m+i)}_x =\mathscr{V}_{\alpha^{i}(x)}^{(m)} \otimes \mathscr{V}^{(i)}_x\] 
is given by  $v\mapsto \varsigma^{(i)}(x)\otimes v$ for $v\in \mathscr{V}^{(i)}_x$. On $\overline{Y_k} \setminus Y_k$, all $\varsigma^{(i)}$'s are 0 by assumption of $\varsigma$.
For notational convenience, set 
\[Y_{[i,j]}:=\overline{Y_i\cup\cdots \cup Y_j}, \ \ i\leq j.\]   
Note that there is $j_m\in \{1,\dots, k\}$ for which  \[r_{j_m} -1< r_k -1- m \leq r_{j_m+1}-1.\]
In particular, if $m=0$, then $r_{j_m+1}=r_k$ (or $j_m=k-1$). Consider the following closed subset 
\begin{align*}
\mathcal C
    =&\, \cup_{i=0}^{r_1-1}\alpha^i(Y_{[1,k]})\,\cup\, \cup_{i=r_1}^{r_2-1}\alpha^i(Y_{[2,k]})\,\cup\, \cdots \,\cup\, \cup_{i=r_{j_m}}^{r_k-1-m}\alpha^i(Y_{[j_m+1,k]})\\
    &\, \cup\, \cup_{i=1}^m\alpha^{-i}(Y)
\end{align*}
of $X$. 

We will define a section $\widetilde\xi\in \Gamma(\mathscr V^{(m)})$ such that $\eta:=\psi^{-1}(\widetilde\xi)$ satisfies the claim, where $\psi$ is the isomorphism of Proposition~\ref{EandSections}. 
For this, we first define a local section $\xi: \mathcal C\to \mathscr V^{(m)}|_{\mathcal C}$ and then extend it to $\widetilde\xi$ over the whole space $X$. Since the closed subsets $\alpha^i(Y_{[l,k]})$, $1\leq l\leq j_m+1$, $r_{l-1}\leq i\leq r_l-1$,  of $\mathcal C$ have overlaps (and contain $\alpha^i(\overline{Y_k})$),  we need  $\xi$ to be well defined on any possible overlaps. 

Define, for $x\in \mathcal C$,  
 \[
\xi(x): = \begin{cases}
	\varsigma^{(i)}\circ \alpha^{-i}(x), &\text{for $x\in \alpha^i(\overline{Y_k}) $},\\
	0, &\text{otherwise.}
\end{cases}
\] 
In order to check that $\xi(x)$ is well-defined on any possible overlaps, first let \[x=\alpha^i(z_1)=\alpha^j(z_2)\in \alpha^i(\overline{Y_k})\cap \alpha^j(\overline{Y_k})\] for some $z_1, z_2 \in \overline{Y_k}$ and $i<j$. Then 
$\alpha^{j-i}(z_2)=z_1\in Y$ while $0<j-i < r_k$. Thus   $z_2 \notin Y_k$, and so  $ z_2 \in \overline{Y_k} \setminus Y_k$. Then 
$\varsigma^{(j)}\circ\alpha^{-j}(x)=\varsigma^{(j)}(z_2)=0$. Also
$\alpha^{r_k-j+i}(z_1) =\alpha^{r_k}(z_2) \in Y
$
implies $z_1 \notin Y_k$ since $0<r_k-j+i<r_k$. Then $z_1\in \overline{Y_k}\setminus Y_k$ and similarly we have $\varsigma^{(i)}\circ\alpha^{-i}(x)=\varsigma^{(i)}(z_1)=0$. Thus $\xi$ is well defined (to be $0$) on overlaps of the form $\alpha^i(\overline{Y_k})\cap \alpha^j(\overline{Y_k})$. 
To show that $\xi$ is well defined to be $0$ on an intersection of the form  
\[
\alpha^i(\overline{Y_k}) \cap \big(\cup_{j=1}^m \alpha^{-j}(Y)\big)
\] for $0\leq i \leq r_k-1-m$, suppose that $x=\alpha^i(\overline{y_k})\in \cup_{j=1}^m \alpha^{-j}(Y)$ for some $\overline{y_k}\in \overline{Y_k}$. Then $\alpha^{j+i}(\overline{y_k})\in Y$ for some $j$, but $0\leq j+i\leq r_k-1$ hence $\overline{y_k}\notin Y_k$.   Then $\overline{y_k}\in \overline{Y_k}\setminus Y_k$ again, and then $\varsigma^{(i)}\circ \alpha^{-i}(x)=\varsigma^{(i)}(\overline{y_k})=0$. 
 Now suppose $\alpha^p(\overline{Y_q})\cap \alpha^i(Y_k)\neq \emptyset$ for some $1\leq q<k$,  $0\leq p<r_q$, $0\leq i<r_k$, with an element $x=\alpha^p(\overline{y_q})=\alpha^i(y_k)$ where  $\overline{y_q}\in \overline{Y_q}$ and $y_k\in Y_k$. Then $p\geq i$ by (\ref{ip}), and $\alpha^{r_q-p+i}(y_k)=\alpha^{r_q-p}(x)=\alpha^{r_q-p}(\alpha^p(\overline{y_q}))=\alpha^{r_q}(\overline{y_q})\in Y$, which is not possible because $0<r_q-p+i<r_k$ and $y_k\in Y_k$. Thus $\alpha^p(\overline{Y_q})\cap \alpha^i(\overline{Y_k})=\alpha^p(\overline{Y_q})\cap \alpha^i(\overline{Y_k}\setminus Y_k)$ on which $\varsigma^{(i)}\circ \alpha^{-i}=0$. 
 Thus $\xi$ is well defined over the closed subset $\mathcal C$ of $X$.  
 
Using the Tietze extension theorem for vector bundles, we further extend the domain of $\xi$ to the whole $X$ to obtain a section $\widetilde{\xi}\in \Gamma(\mathscr V^{(m)})$ over $X$ such that 
\[
\widetilde{\xi}|_{\alpha^i(\overline{Y_q})}=\xi|_{\alpha^i(\overline{Y_q})}=0\]  for $ 1 \leq q \leq k-1 $, $0\leq i\leq \min\{r_q-1, r_k-1-m\}$,  while
\[
\widetilde{\xi}|_{\alpha^i(\overline{Y_k})}=\xi|_{\alpha^i(\overline{Y_k})},
\]  for  $0\leq i\leq r_k-1-m$,  
and such that
\[
\widetilde{\xi}|_{\cup_{j=1}^m \alpha^{-j}(Y)}=0.
\] 
Under the isomorphism $\psi:\E^{\otimes m}= \Gamma(\mathscr V,\alpha)^{\otimes m}\to \Gamma(\mathscr V^{(m)},\alpha^m)$, the property $\widetilde{\xi}|_{\cup_{j=1}^m \alpha^{-j}(Y)}=0$ implies that $\eta=\psi^{-1}(\widetilde \xi)\in \E_Y^{\otimes m}\cong E_m\subset \mathcal O(\E_Y)$.  If $m=0$, then  $\cup_{j=1}^m \alpha^{-j}(Y):=\emptyset$,  $\E^{\otimes 0}= \Gamma(\mathscr V^{(0)},\alpha^0)=C(X)$ and $\psi$ is the identity $^*$-homomorphism of $C(X)$. 

Finally, for  $x \in \overline{Y_i}$, $1\leq i\leq k-1$,  and a vector $(a_0, a_1, \ldots, a_{r_i-1})\in \mathscr D^{(r_i)}_x=\mathscr V^{(0)}_x\oplus\cdots\oplus \mathscr V^{(r_i-1)}_x$,  we have 
\begin{align*}
&\  \pi_i(\psi^{-1}(\widetilde \xi))(x)(a_0, a_1, \ldots, a_{r_i-1}) \\
=&\ ( 0,\ldots,0, \psi(\psi^{-1}(\widetilde \xi))(x)\otimes a_0, \psi(\psi^{-1}(\widetilde \xi))(\alpha(x))\otimes a_1,\\
&\qquad\qquad\qquad\qquad  \ldots, \psi(\psi^{-1}(\widetilde \xi))(\alpha^{r_i-1-m}(x)) \otimes a_{r_i-1-m})\\
=&\ ( 0,\ldots,0, \widetilde \xi(x)\otimes a_0,  \widetilde \xi(\alpha(x))\otimes a_1, \ldots,  \widetilde \xi(\alpha^{r_i-1-m}(x)) \otimes a_{r_i-1-m})\\
=&\ ( 0,\ldots,0, 0,  \ldots,  0) =0,
\end{align*}
and  for $x \in \overline{Y_k}$, $\pi_k(\psi^{-1}(\widetilde \xi))= \varsigma$ (see Remark~\ref{remark:identification}). 
Hence $\pi_i(\eta)=0$ for $1\leq i\leq ,k-1$ and $\pi_k(\eta)=\varsigma$ as desired. 
\end{proof}

Observe that the structure of the matrix bundles $\mathscr{M}_k$ in the iterated pullback 
\begin{equation}
    \Gamma(\mathscr{M}_1) \oplus_{\Gamma(\mathscr{M}_2|_{\overline{Y_2}\setminus Y_2})} \oplus \dots \oplus_{\Gamma(\mathscr{M}_K|_{\overline{Y_K}\setminus Y_K})} \Gamma(\mathscr{M}_K),
\end{equation} 
is dictated by the line bundle $\mathscr{V}$. In particular, if $\mathscr{V}$ is trivial, then so too are the $\mathscr{M}_k$. In that case $\Gamma(\mathscr{M}_k) \cong C(\overline{Y_k}, M_{r_k})$ and we obtain the same decomposition as the one given in \eqref{rsh-trivial}. 

In the case of a non-trivial line bundle, we arrive at the following.

\begin{theorem}\label{maintheorem}
 Let $X$ be an infinite compact metric space, $\alpha : X \to X$ a minimal homeomorphism, $\mathscr{V}$ a line bundle over $X$ and $Y \subset X$ a closed subset with non-empty interior. Let $r_1 < \dots < r_K$ denote the distinct first return times to $Y$. The orbit breaking algebra $\mathcal O (C_0(X \setminus Y)\Gamma(\mathscr{V}, \alpha))$ has a recursive subhomogeneous decomposition $B$ such that 
 \begin{enumerate}
     \item the length $l$ of $B$ is at least $K$,
     \item the base spaces $Z_1, \dots, Z_l$ have dimension bounded by $X$, and there is $l_0= 0 < l_1 < l_2 < \dotsm < l_K = l$ such that 
     \[
    \bigcup_{j=l_{k-1}+1}^{l_k} Z_j = \overline{\{ y \in Y \mid r(y) = r_k\}} 
     \]
     for every $1 \leq k \leq K$,
     \item the matrix sizes are exactly $r_1, \dots, r_K$.
 \end{enumerate}
\end{theorem}

\begin{proof} Let $\E_Y := C_0(X\setminus Y)\Gamma(\mathscr{V}, \alpha)$. By the previous theorem, we have that  
\begin{equation*}
 \mathcal{O}(\E_Y) \cong [\cdots [\Gamma(\mathscr{M}_1)\oplus_{\Gamma(\mathscr{M}_2|_{\overline{Y}_2 \setminus Y_2})} \Gamma(\mathscr{M}_2)]\oplus \cdots ]\oplus_{\Gamma(\mathscr{M}_K|_{\overline{Y}_K \setminus Y_K})} \Gamma(\mathscr{M}_K).
 \end{equation*}
 By \cite[Proposition 1.7]{Phillips:recsub}, $ \Gamma(\mathscr M_k)$, for every $1 \leq k \leq K$, is an RSH algebra with matrix sizes $r_k$ with base spaces $Z_{l_{k-1} + 1}, \dots, Z_{l_k}$ for some $l_k \geq l_{k-1} + 1$ whose union is the base space of the bundle, that is 
 \[ 
    \bigcup_{j=l_{k-1}+1}^{l_k} Z_j = \overline{\{ y \in Y \mid r(y) = r_k\}},
     \] 
It follows from \cite[Proposition 3.2]{Phillips:recsub} that $\mathcal{O}(\E_Y)$ itself has an RSH decomposition with length at least $K$, whose bases spaces are bounded in dimension by the dimension of $X$ and satisfy (2), and whose matrix sizes are $r_1, \dots, r_K$.
\end{proof}

A main motivation for the RSH construction is to provide classification results for the $\mathrm C^*$-algebras associated to minimal homeomorphisms twisted by line bundles and their simple orbit-breaking subalgebras.  In~\cite{FJS}, the authors will show that classification is possible under the assumption that the homeomorphism has mean dimension zero. However, Theorem~\ref{maintheorem} is already enough to give an alternative proof of the classification results from \cite{AAFGJSV2024}, which is to say, the case where $\dim X< \infty$.

Let $A$ be a simple, separable, nuclear, unital $\mathrm C^*$-algebra. The \emph{Elliott invariant} of $A$, denoted $\Ell(A)$ is given by
 \[ \Ell(A) = (K_0(A), K_0(A)_+, [1_A], K_1(A), T(A), \rho),\]
 where $(K_0(A), K_0(A)_+, [1_A], K_1(A))$ is the (pointed, ordered) $K$-theory of $A$, $T(A)$ the tracial state simplex, and $\rho : K_0(A) \times T(A) \to \mathbb{R}$ the pairing map defined by $\rho([p]-[q], \tau) = \tau(p) - \tau(q)$, for $\tau$ the (non-normalized) inflation of a tracial state to a suitable matrix algebra over $A$. 

\begin{corollary}[cf. {\cite[Theorem 6.18]{AAFGJSV2024}}] \label{cor:mainthmCor} Let $X$ and $Y$ be finite-dimensional infinite compact metric spaces. Let $\E = \Gamma(\mathscr{V}, \alpha)$ and $\F = \Gamma(\mathscr{W}, \beta)$ be a Hilbert $C(X)$-module and $C(Y)$-module, respectively, with respect to the line bundles $\mathscr{V}$ over $X$, $\mathscr{W}$ over $Y$, and minimal homeomorphisms $\alpha : X \to X$, $\beta: Y \to Y$. Then 
\[ \mathcal{O}_{C(X)}(\E) \cong \mathcal{O}_{C(Y)}(\F) \]
    if and only if 
    \[ \Ell(\mathcal{O}_{C(X)}(\E)) \cong \Ell(\mathcal{O}_{C(Y)}(\F)).\]
\end{corollary}

\begin{proof}
    Let $y \in X$ be a single point, and let $Y_1 \supset Y_2 \supset Y_3 \supset \cdots$ be a sequence of closed subsets of $X$, each with non-empty interior, satisfying 
    \[ \cap_{n \in \mathbb{Z}_{>0}} Y_n = \{ y\}.\]
    Then
    
    \[ \mathcal{O}_{C(X)}(\E_{\{y\}}) \cong \dlim \, \mathcal{O}_{C(X)}(\E_{Y_n}).\]
    By Theorem~\ref{maintheorem}, the orbit-breaking algebras $\mathcal{O}(\E_{Y_n})$ are recursive subhomogeneous $\mathrm C^*$-algebras whose base spaces are bounded in dimension by the dimension of $X$. By \cite[Theorem 6.1]{Winter:subhomdr}, it follows that $\mathcal{O}_{C(X)}(\E_{Y_n})$ has decomposition rank bounded by $\dim X$. Since $\mathcal{O}_{C(X)}(\E_{\{y\}})$ is the inductive limit of the $\mathcal{O}_{C(X)}(\E_{Y_n})$, the decomposition rank $\mathcal{O}_{C(X)}(\E_{\{y\}})$ is also bounded by $\dim X$ \cite[3.3 (ii)]{KirWinter:dr}. It follows that $\mathcal{O}_{C(X)}(\E_{\{y\}})$ is $\mathcal{Z}$-stable \cite[Theorem 5.1]{Winter:dr-Z-stable}. Since $\mathcal{O}(\E_{\{y\}})$ is a centrally large subalgebra of $\mathcal{O}_{C(X)}(\E)$ \cite[Theorem 6.16]{AAFGJSV2024}, it too is $\mathcal{Z}$-stable. Similarly, $\mathcal O_{C(Y)}(\F)$ is $\mathcal Z$-stable. Thus $\mathcal O_{C(X)}(\E)$ and $\mathcal O_{C(Y)}(\F)$ are simple, separable, unital, nuclear, $\mathcal Z$-stable $\mathrm{C}^*$-algebras and the result follows from \cite{ElliottGong2025}.
\end{proof}

\begin{remark}
  \begin{enumerate}[left=0pt]
  \item If the bundle $\mathscr{V}$ is trivial, then the RSH decomposition of Theorem~\ref{maintheorem} is the same as the RSH decomposition for $\mathrm C^*(X, C_0(X \setminus Y) \subset C(X) \rtimes_\alpha \mathbb{Z}$ given in \cite{QLin:Ay}.
\item If $(X, \alpha)$ is only assumed to have mean dimension zero (which means $\dim X$ may be infinite), we can't deduce finite decomposition rank of $\mathcal{O}(\E_{\{y\}})$, which was used in the proof of Corollary~\ref{cor:mainthmCor} to show classification. This is because the base spaces of the RSH decomposition for each $\mathcal{O}(\E_{Y_n})$ can be infinite dimensional.  However, using the RSH decomposition of $C^*(C(X), C_0(X \setminus Y)u) \subset C(X)\rtimes_\alpha \mathbb{Z}$ for $Y \subset X$
a closed subset with non-empty interior, Elliott and Niu show that for any point $y \in X$, the orbit-breaking algebra  $C^*(C(X), C_0(X \setminus \{y\})u) \subset C(X)\rtimes_\alpha \mathbb{Z}$ can be locally approximated by subhomogeneous $\mathrm{C}^*$-algebras with arbitrarily small dimension ratio \cite{EllNiu:MeanDimZero}, which is enough to imply $\mathcal{Z}$-stability (and hence classifiability) of both the orbit-breaking algebra $C^*(C(X), C_0(X \setminus \{y\})u)$ and its containing crossed product $C(X)\rtimes_\alpha \mathbb{Z}$. In \cite{FJS}, we will use the RSH construction of Theorem~\ref{maintheorem} to generalize this to the Cuntz--Pimsner algebras $\mathcal O(\Gamma(\mathscr V, \alpha))$ and $\mathcal O(\Gamma(\mathscr V, \alpha)_{\{ y\}})$.

\item Let $(X, \alpha)$ be a minimal system. In \cite{alboiu-lutley}, it is shown that crossed product $C(X)\rtimes_\alpha \mathbb{Z}$ has stable rank one, with no restrictions an on $(X, \alpha)$. In particular, $C(X)\rtimes_\alpha \mathbb{Z}$ may have stable rank one even if it is not $\mathcal{Z}$-stable. This is shown by exhibiting a so-called \emph{diagonal subhomogeneous} (DSH) algebra decomposition for orbit-breaking subalgebras when $Y \subset X$ is closed with non-empty interior. A DSH algebra is a special type of RSH algebra where the gluing maps between stages of the RSH decomposition are assumed to be \emph{diagonal} (see \cite[Definition 2.3]{alboiu-lutley} for a precise definition). As above, if $y$ is a single point, then this implies $\mathrm{C}^*(X, C_0(X \setminus \{y\})u) \subset C(X) \rtimes_\alpha \mathbb{Z}$ is approximately DSH. They show that a simple inductive limit of DSH algebras with diagonal connecting maps has stable rank one \cite[Theorem 3.30]{alboiu-lutley}. Thus $\mathrm{C}^*(X, C_0(X \setminus \{y\})u)$ has stable rank one, and as it is a centrally large subalgebra, this passes to $C(X)\rtimes_\alpha \mathbb{Z}$. The RSH constructions in this paper will \emph{not} be DSH algebra in general. This follows from \cite[Corollary 2.22]{alboiu-lutley}. Nevertheless, our RSH constructions should be amenable to similar techniques. The question of stable rank of $\mathrm{C}^*$-algebras associated to homeomorphisms twisted by (non-trivial) line bundles is the subject of future investigations. 
  \end{enumerate} 
\end{remark}



\end{document}